\documentclass[10pt]{amsart}
\usepackage[style=alphabetic]{biblatex} 
\usepackage{tikz}
\usepackage{amssymb}
\usepackage{amsthm}
\usepackage[all]{xy}
\usepackage{xcolor}
\usepackage{thmtools}
\usepackage{hyperref}
\usepackage[nameinlink]{cleveref}
\usepackage{graphicx}
\usepackage{mathrsfs}
\usepackage{tikz-3dplot}

\makeatletter
    
    \@addtoreset{equation}{section}
  \makeatother

\hypersetup{
 colorlinks,
 linkcolor={teal},
 citecolor={teal},
 urlcolor={teal}
}

\calclayout

\DeclareFieldFormat
  [article,book,inbook,incollection,inproceedings,patent,thesis,unpublished]
  {title}{\emph{#1\isdot}}

\theoremstyle{plain}
	\newtheorem{theorem}{Theorem}[section]
	\newtheorem{lemma}[theorem]{Lemma}
	\newtheorem{corollary}[theorem]{Corollary}
	
	\newtheorem{proposition}[theorem]{Proposition}

\theoremstyle{definition} 
	\newtheorem{remark}[theorem]{Remark}
	\newtheorem{definition}[theorem]{Definition}
	\newtheorem{example}[theorem]{Example}



\newcommand{\Ch}{\mathrm{Ch}}
\newcommand{\Flag}{\mathrm{Flag}}
\newcommand{\zero}{\mathbf{0}}
\newcommand{\one}{\mathbf{1}}
\newcommand{\NN}{\mathbb{N}}
\newcommand{\ZZ}{\mathbb{Z}}

\newcommand{\Qq}{\mathbb{Q}[\![q]\!]}
\newcommand{\ol}[1]{\overline{#1}}
\DeclareMathOperator{\supp}{supp}
\DeclareMathOperator{\rank}{rank}
\DeclareMathOperator{\codim}{codim}
\DeclareMathOperator{\sgn}{sgn}
\newcommand{\A}{\mathcal{A}}
\newcommand{\B}{\mathcal{B}}
\newcommand{\Tpos}{\mathsf{T}}
\newcommand{\Npos}{\mathsf{N}}
\newcommand{\Cset}{\mathcal{C}}
\newcommand{\Tope}{\mathcal{T}}
\newcommand{\Face}{\mathcal{F}}

\newcommand{\MC}{\mathrm{MC}}
\newcommand{\MH}{\mathrm{MH}}
\newcommand{\Br}{\mathrm{Br}}
\newcommand{\Bl}{\mathrm{Bl}}

\newcommand{\Mag}{\mathrm{Mag}}
\newcommand{\wt}{\mathrm{wt}}
\newcommand{\Cay}{\mathrm{Cay}}
\newcommand{\adj}{\mathrm{adj}}

\begin{document}

\title{Magnitude homology of tope graphs}
\author[J. Koizumi]{Junnosuke Koizumi}
\address{RIKEN iTHEMS, Wako, Saitama 351-0198, Japan}
\email{junnosuke.koizumi@riken.jp}

\date{\today}
\subjclass[2020]{Primary 52C35; Secondary 05E45, 51F15, 55N35}

\begin{abstract}
We completely determine the magnitude homology of tope graphs of real hyperplane arrangements.
Their ranks can be described as the Hilbert functions of the Stanley--Reisner rings of certain simplicial complexes naturally associated with the arrangements.
For Coxeter arrangements, this gives a computation of the magnitude homology of the Cayley graph of the corresponding Coxeter group.
We also prove the homological reciprocity for central arrangements conjectured by Koizumi--Liu.
The proof combines poset combinatorics, the Edelman--Walker theorem, and Alexander duality.
\end{abstract}

\maketitle
\setcounter{tocdepth}{1}
\tableofcontents

\enlargethispage*{20pt}
\thispagestyle{empty}

\section{Introduction}

Magnitude is an invariant of metric spaces, or more generally of enriched categories, introduced by Leinster \cite{Leinster2013,Leinster2019}.
For a finite connected graph, its magnitude is a formal power series in $q$ whose constant term is the number of vertices.
Magnitude has various cardinality-like properties, such as additivity and multiplicativity, and may be regarded as an extension of cardinality that reflects metric structure.
Hepworth and Willerton \cite{Hepworth2017} constructed a homology theory called magnitude homology and showed that the coefficient of $q^\ell$ in the magnitude is the Euler characteristic of the length-$\ell$ magnitude homology.
Since then, the magnitude homology of graphs has been studied using various tools from combinatorics and topology, including order complexes \cite{Kaneta2021}, simplicial complex models \cite{Asao2021}, and Morse-theoretic methods \cite{Gu2018,Tajima2023}.
These studies have not only developed various computational methods but also revealed structural properties such as the existence of torsion \cite{Kaneta2021, SazdanovicSummers} and the relationship between diagonality and girth \cite{AHK}.

Because magnitude refines cardinality, it is natural to seek magnitude refinements of enumeration problems involving objects equipped with metrics or graph structures.
In joint work with Liu, the author studied such a ``magnitude refinement'' of the classical problem of counting the chambers of a real hyperplane arrangement \cite{Koizumi_Liu}.
For a real hyperplane arrangement $\A$, the tope graph $\Tope(\A)$ has the chambers as vertices, with edges joining adjacent chambers.
The author and Liu proved that when $\A$ is central, the magnitude and magnitude homology of this graph admit rich combinatorial structures, including a decomposition indexed by faces.
The purpose of this paper is to develop these results further and completely determine the magnitude homology of tope graphs.

In fact, there are still relatively few examples of graphs whose magnitude homology has been completely determined in every bidegree. Known examples include trees and complete graphs \cite{Hepworth2017}, cycles \cite{Gu2018}, certain outerplanar graphs \cite{SazdanovicSummers}, and geodetic graphs \cite{Asao_Wakatsuki}.
These existing computations relied either on decomposing graphs into simpler ones using Mayer--Vietoris sequences or the K\"unneth formula, or on the simplicity of their geodesic structures.
Since these existing methods do not appear to apply directly to tope graphs, we use a different approach in our computations.

Our main result can be described in terms of combinatorial objects called face flags of a hyperplane arrangement.
A face flag relative to a chamber $B$ is a weakly decreasing sequence
\[
\mathscr{F}=(F_1\geq F_2\geq \cdots\geq F_m)
\]
of faces of $\overline{B}$.
Each face flag $\mathscr{F}$ is assigned a bidegree $(k,\ell)$; see \Cref{sec:affine_closed}. Denote the set of face flags of bidegree $(k,\ell)$ by $\Flag_{k,\ell}(\A)_B$, and define $\Flag_{k,\ell}(\A)$ to be their disjoint union over all chambers $B$.
Then our formula is
\[
\MH_{k,\ell}(\Tope(\A))
\cong
\ZZ^{\oplus \Flag_{k,\ell}(\A)}.
\]
In particular, the magnitude homology of a tope graph is torsion-free.
If $\A$ is finite, we also show that the ranks of these groups are determined by the intersection poset.
This proves parts (4) and (5) of \cite[Conjecture 7.1]{Koizumi_Liu}.
Moreover, applying our result to the Coxeter arrangements associated with Coxeter systems of finite or affine type yields an explicit formula for the magnitude homology of the Cayley graphs of the corresponding Coxeter groups.

Interestingly, the magnitude Betti numbers described above can also be expressed through the Hilbert functions of Stanley--Reisner rings of simplicial complexes naturally associated with the arrangement.
For each chamber $B$, let $P_B$ be the poset of proper faces of $\overline{B}$, and let $R_{\A,B}$ be the Stanley--Reisner ring of $\Delta P_B$. Assign to the variable $x_F$ corresponding to a face $F$ the bidegree $(k,\ell)$ determined by $F$.
If $\Tope(\A)$ is of finite, we obtain
\[
\sum_{k,\ell\geq0}\rank\MH_{k,\ell}(\Tope(\A))u^kq^\ell
=
\sum_{B\in \Ch(\A)} \operatorname{Hilb}(R_{\A,B};u,q).
\]
This identity suggests a connection between magnitude theory and combinatorial commutative algebra.

The starting point of our proof is the observation that a tope graph carries structure richer than its metric alone: one can record the number of times each hyperplane is crossed. Correspondingly, magnitude homology also admits a direct-sum decomposition:
\[
    \MH_{\ast,\ell}(\Tope(\A))=
    \bigoplus_{\substack{\alpha\in \NN^{\oplus \A}\\|\alpha|=\ell}} \MH_{\ast,\alpha}(\Tope(\A)).
\]
First, we prove the localization theorem for these direct summands: if $\A_\alpha=\supp(\alpha)$, then $\MH_{k,\alpha}(\Tope(\A))$ ``localizes'' to the subarrangement $\A_\alpha$.
Second, we prove a periodicity theorem for central hyperplane arrangements:
if $\A$ is a central hyperplane arrangement of rank $r$ and $\alpha\geq\one$, then
\[
        \MH_{k,\alpha}(\Tope(\A))
        \cong
        \MH_{k-r,\alpha-\one}(\Tope(\A)).
\]
The proof of this periodicity uses an approach based on the topology of order complexes of posets and Alexander duality.
Finally, we show that the recursion for face flags matches the recursion for magnitude homology obtained by applying localization and periodicity.

For central hyperplane arrangements, we also obtain the homological reciprocity theorem, which resembles other combinatorial reciprocity theorems such as the Ehrhart--Macdonald reciprocity.
If $\A$ is a central hyperplane arrangement of rank $r$ and $n=\#\A$, then we have
\[
        \MH_{k,\ell}^{\circ}(\Tope(\A))
        \cong
        \MH_{k-r,\ell-n}(\Tope(\A)).
\]
Here, the left-hand side is the interior magnitude homology defined by Koizumi--Liu \cite{Koizumi_Liu}.
This resolves part (6) of \cite[Conjecture 7.1]{Koizumi_Liu}.

Finally, we apply our formula to concrete examples.
The formula recovers Gu's computation of the magnitude homology of even cycles \cite{Gu2018}.
It also gives a complete computation for the Cayley graph of $S_n$ with respect to adjacent transpositions.
We also construct, for every $n\geq 1$, a graph whose magnitude homology is concentrated on the diagonal for $k\leq n$, but is not diagonal in all degrees.

The paper is organized as follows.
Section~2 recalls the necessary material from magnitude homology and affine hyperplane arrangements.
Section~3 proves the localization theorem and states the periodicity theorem.
Sections~4 and 5 prove the periodicity theorem.
Section~6 proves the main theorem, gives an interpretation in terms of Stanley--Reisner rings, and discusses the case of Coxeter arrangements.
Section~7 treats examples, including even cycles and the Cayley graph of $S_n$.

\subsection*{Use of AI}
In writing this paper, we used AI in the following ways.
\begin{itemize}
    \item Some ideas used in the proofs of \Cref{lem:homological-quillen}, \Cref{lem:nonfacial_acyclic}, \Cref{lem:nonfacial_criterion}, and \Cref{prop:strong_acyclicity} were suggested by ChatGPT-5.5 Pro.
    \item We used ChatGPT-5.5 Pro to improve the prose and to identify grammatical and mathematical errors.
\end{itemize}
All AI-generated suggestions and text incorporated into the paper were carefully checked by the author.
The author takes full responsibility for the text of this paper.

\subsection*{Conventions and notation}
We write $\NN=\{0,1,2,\ldots\}$.
If $E$ is a set and $\alpha\in\NN^{\oplus E}$, we write
\[
        \supp(\alpha):=\{e\in E\mid \alpha_e>0\}
        \quad\text{and}\quad
        |\alpha|:=\sum_{e\in E}\alpha_e.
\]
For $E'\subseteq E$, the restriction of $\alpha$ to $E'$ is denoted by
$\alpha|_{E'}\in\NN^{\oplus E'}$.
When $E$ is finite, we write $\one_E\in\NN^{\oplus E}$ for the vector all of whose coordinates are equal to
$1$, and we abbreviate it to $\one$ when the ground set is clear.
Similarly, $\zero$ denotes the zero vector on the relevant ground set.
Thus $\supp(\alpha)=\varnothing$ is equivalent to $\alpha=\zero$.
The order on $\NN^{\oplus E}$ is always the coordinatewise order: for
$\alpha,\beta\in \NN^{\oplus E}$,
\[
        \alpha\leq \beta
        \quad\Longleftrightarrow\quad
        \forall e\in E,\quad \alpha_e\leq \beta_e.
\]
For a subposet $P\subseteq \NN^{\oplus E}$ and $x\in\NN^{\oplus E}$, put
$P_{\leq x}:=\{p\in P\mid p\leq x\}$.
For $\alpha\in\NN^{\oplus E}$, let
\[
        \ol\alpha:=(\alpha\bmod 2)\in \{0,1\}^E
\]
be its coordinatewise parity vector.

Graphs are assumed to be simple and undirected.
Hyperplane arrangements are assumed to be real.
Topological statements about posets are interpreted through their order
complexes.
For example, a poset $P$ is called \emph{acyclic} if
\[
        \forall i\in \ZZ,\quad \widetilde H_i(\Delta P)=0.
\]
We use the standard convention that the empty complex has reduced homology
$\widetilde H_{-1}(\varnothing)=\ZZ$ and no other reduced homology.
Thus the empty poset is not acyclic, while every contractible poset is acyclic.

\section{Preliminaries}

We first recall basic concepts and results about magnitude homology and hyperplane arrangements.

\subsection{Magnitude of $\NN$-valued metric spaces}
An $\NN$-valued metric space is a set $X$ equipped with a map $d\colon X\times X\to \NN$ satisfying the usual metric axioms.
For example, if $G$ is a connected graph, possibly infinite, then $V(G)$ can be regarded as an $\NN$-valued metric space via the shortest-path metric.
We denote the resulting $\NN$-valued metric space simply by $G$.

Let $X$ be an $\NN$-valued metric space.
We say that $X$ is \emph{of finite type} if for every $x\in X$ and $r\geq 0$, the set $B(x,r)=\{y\in X\mid d(x,y)\leq r\}$ is finite.
In this situation, for any $w \colon X\to\Qq$ and $x\in X$, the sum
\[
\sum_{y\in X}q^{d(x,y)}w(y)
\]
converges in $\Qq$, since this becomes a finite sum modulo $q^n$ for every $n\geq 0$.
A \emph{magnitude weighting} for $X$ is a map $\wt \colon X\to \Qq$ satisfying
\[
\forall x\in X,\quad \sum_{y\in X}q^{d(x,y)}\wt(y)=1.
\]

A proper $k$-chain in $X$ is a $(k+1)$-tuple $\vec{x}=(x_0,x_1,\ldots,x_k)\in X^{k+1}$ such that $x_i\neq x_{i+1}$ for all $i$.
Define its length by
\[
\ell(\vec{x})=d(x_0,x_1)+d(x_1,x_2)+\cdots+d(x_{k-1},x_k).
\]
We write $P_{k}(X)$ for the set of proper $k$-chains in $X$, and set
\begin{align*}
P_{k,\ell}(X)&=\{\vec{x}\in P_k(X)\mid \ell(\vec{x})=\ell\},\\
P_{k,\ell}(X)_{a\to\bullet}&=\{\vec{x}\in P_k(X)\mid \ell(\vec{x})=\ell,\ x_0=a\},\\
P_{k,\ell}(X)_{\bullet\to b}&=\{\vec{x}\in P_k(X)\mid \ell(\vec{x})=\ell,\ x_k=b\},\\
P_{k,\ell}(X)_{a\to b}&=\{\vec{x}\in P_k(X)\mid \ell(\vec{x})=\ell,\ x_0=a,\ x_k=b\}
\end{align*}
for $a,b\in X$.
Note that $P_{k,\ell}(X)=\varnothing$ for $k>\ell$.
If $X$ is of finite type, then $P_{k,\ell}(X)_{a\to\bullet}$ and $P_{k,\ell}(X)_{\bullet\to b}$ are finite sets.

\begin{lemma}\label{lem:mag_expansion}
    Let $X$ be an $\NN$-valued metric space of finite type.
    Then there exists a unique magnitude weighting $\wt\colon X\to \Qq$ for $X$, which is given by
    \[
    \wt(b)=
    \sum_{\ell=0}^\infty
    \sum_{k=0}^\ell
    (-1)^k \#P_{k,\ell}(X)_{\bullet\to b}\,q^\ell.
    \]
\end{lemma}

\begin{proof}
    See \cite[Corollary 3.23]{Asao_filtered}.
\end{proof}

When $X$ is finite, we define the \emph{magnitude} of $X$ to be the total sum of its magnitude weighting:
\[
\Mag(X)=\sum_{x\in X}\wt(x)\in \Qq.
\]
Note that we have $\Mag(X)|_{q=0}=\#X$, since $\wt(x)|_{q=0}=1$ for every $x\in X$.
Therefore, magnitude may be regarded as a natural refinement of cardinality.

\subsection{Magnitude homology of $\NN$-valued metric spaces}
Let $X$ be an $\NN$-valued metric space.
We define $\MC_{k,\ell}(X)$ to be the free abelian group generated by the set $P_{k,\ell}(X)$ of proper $k$-chains of length $\ell$.
We define the boundary operator $\partial:\MC_{k,\ell}(X)\to \MC_{k-1,\ell}(X)$ by $\partial=\sum_{i=1}^{k-1}(-1)^i\partial_i$, where
\[
    \partial_i(x_0,\ldots,x_k)=\begin{cases}
    (x_0,\ldots,\widehat{x_i},\ldots,x_k), & \text{ if } d(x_{i-1},x_{i+1})=d(x_{i-1},x_i)+d(x_i,x_{i+1});\\
    0, & \text{ otherwise}.
\end{cases}
\]
Here, $\widehat{x_i}$ denotes deletion of $x_i$. It is straightforward to check that $\partial^2=0$.
The homology of the chain complex $(\MC_{\ast,\ell}(X),\partial)$ is called the length-$\ell$ \emph{magnitude homology} of $X$ and is denoted by $\MH_{\ast,\ell}(X)$.
When $\MH_{k,\ell}(X)$ is of finite rank, we define the \emph{magnitude Betti number} of $X$ by $\beta_{k,\ell}(X)=\rank \MH_{k,\ell}(X)$.

Let $a,b\in X$.
We define
\[
\MC_{\ast,\ell}(X)_{a\to\bullet},\quad
\MC_{\ast,\ell}(X)_{\bullet\to b},\quad
\MC_{\ast,\ell}(X)_{a\to b}
\]
to be the subcomplexes of $\MC_{\ast,\ell}(X)$ spanned by
\[
P_{\ast,\ell}(X)_{a\to\bullet},\quad
P_{\ast,\ell}(X)_{\bullet \to b},\quad
P_{\ast,\ell}(X)_{a\to b},
\]
respectively.
Then we have
\[
    \MC_{\ast,\ell}(X)
    \cong \bigoplus_{a\in X}\MC_{\ast,\ell}(X)_{a\to \bullet}
    \cong \bigoplus_{b\in X}\MC_{\ast,\ell}(X)_{\bullet\to b}
    \cong \bigoplus_{a,b\in X}\MC_{\ast,\ell}(X)_{a\to b}.
\]
This induces decompositions of homology groups
\begin{equation}\label{vertexdecomp}
    \MH_{k,\ell}(X)
    \cong\bigoplus_{a\in X}\MH_{k,\ell}(X)_{a\to\bullet}
    \cong\bigoplus_{b\in X}\MH_{k,\ell}(X)_{\bullet\to b}
    \cong\bigoplus_{a,b\in X}\MH_{k,\ell}(X)_{a\to b}.
\end{equation}
If $X$ is of finite type, then $\MH_{k,\ell}(X)_{a\to \bullet}$, $\MH_{k,\ell}(X)_{\bullet\to b}$, and $\MH_{k,\ell}(X)_{a\to b}$ are finitely generated, since $P_{k,\ell}(X)_{a\to \bullet}$, $P_{k,\ell}(X)_{\bullet\to b}$, and $P_{k,\ell}(X)_{a\to b}$ are finite.
The next lemma shows that $\MH_{k,\ell}(X)_{\bullet\to b}$ categorifies the magnitude weighting:

\begin{lemma}\label{lem:weight}
    Let $X$ be an $\NN$-valued metric space of finite type.
    Then the unique magnitude weighting $\wt \colon X\to \Qq$ for $X$ is given by
    \[
    \wt(b)=\sum_{\ell=0}^\infty\sum_{k=0}^\ell(-1)^k\rank \MH_{k,\ell}(X)_{\bullet\to b}\,q^\ell.
    \]
\end{lemma}

\begin{proof}
    By \Cref{lem:mag_expansion}, we have
    \[
    \wt(b)=\sum_{\ell=0}^\infty\sum_{k=0}^\ell(-1)^k\rank \MC_{k,\ell}(X)_{\bullet\to b}\,q^\ell.
    \]
    For each $\ell$, the complex $\MC_{\ast,\ell}(X)_{\bullet\to b}$ is a bounded complex of finitely generated free abelian groups, so the Euler--Poincar\'e identity gives the desired formula.
\end{proof}

We say that an $\NN$-valued metric space $X$ is \emph{diagonal} if $\MH_{k,\ell}(X)=0$ for $k\neq\ell$.
It is known that the complete graph $K_n$ and the hypercube graph $Q_n$ are diagonal \cite{Hepworth2017}.
On the other hand, Gu \cite{Gu2018} computed $\MH_{k,\ell}(C_n)$ and showed that the $n$-cycle $C_n$ is never diagonal for $n\geq 5$.

For $a,b\in X$, define the \emph{closed interval} $[a,b]_X$ by
\[
[a,b]_X=\{x\in X\mid d(a,x)+d(x,b)=d(a,b)\}.
\]
We define a partial order $\preceq_{a,b}$ on $[a,b]_X$ by
\[
        x\preceq_{a,b} y
        \quad\Longleftrightarrow\quad
        d(a,x)+d(x,y)+d(y,b)=d(a,b).
\]
We write $\prec_{a,b}$ for the associated strict order.
The \emph{open interval} $(a,b)_X$ is defined as the subposet obtained from $[a,b]_X$ by removing $a$ and $b$.

\begin{proposition}[{cf. \cite[Proposition 2.3]{Gomi2025}}]\label{intervalMH}
    Let $X$ be an $\NN$-valued metric space.
    For $a,b\in X$ with $d(a,b)=\ell>0$, there is an isomorphism of chain complexes
    \[
    \MC_{\ast,\ell}(X)_{a\to b}\cong \widetilde{C}_{\ast-2}(\Delta(a,b)_X),
    \]
    where the right-hand side is the augmented simplicial chain complex.
    This induces an isomorphism of homology groups
    \[
    \MH_{k,\ell}(X)_{a\to b}\cong\widetilde{H}_{k-2}(\Delta(a,b)_X).
    \]
\end{proposition}

\begin{proof}
    Any proper chain $\vec{x}\in P_{k,\ell}(X)_{a\to b}$ must be geodesic since $\ell=d(a,b)$.
    Mapping $\vec{x}$ to the element $(-1)^k[x_1\prec_{a,b}\cdots\prec_{a,b} x_{k-1}]$ in $\widetilde{C}_{k-2}(\Delta(a,b)_X)$ defines the desired isomorphism.
\end{proof}

\subsection{Magnitude homology of partial cubes}
Let $E$ be a set.
The \emph{Hamming distance} on the hypercube $\{+,-\}^E$ is defined by
\[
d_H(v,w)=\#\{e\in E\mid v_e\neq w_e\} \in \NN\cup\{\infty\}.
\]
We say that an $\NN$-valued metric space $(X,d)$ is a \emph{partial cube} if there exists an isometric embedding $i\colon X\hookrightarrow\{+,-\}^E$ for some set $E$.
Fixing such an embedding, we can define, for two points $x,y\in X$, the distance vector $\vec{d}(x,y)\in \NN^{\oplus E}$ by
\[
\vec{d}(x,y)_e=
\begin{cases}
1,&\text{if }i(x)_e\neq i(y)_e;\\
0,&\text{if }i(x)_e=i(y)_e.
\end{cases}
\]
This refines the distance function because $\lvert\vec d(x,y)\rvert=d(x,y)$.
This refinement gives rise to a decomposition of magnitude homology, as explained below.

Let $X$ be a partial cube and fix an isometric embedding $i\colon X\hookrightarrow \{+,-\}^E$.
For a proper $k$-chain $\vec{x}\in P_k(X)$, we define its \emph{length vector} by
\[
\vec{\ell}(\vec{x})=\vec{d}(x_0,x_1)+\vec{d}(x_1,x_2)+\cdots+\vec{d}(x_{k-1},x_k) \in \NN^{\oplus E}.
\]
For $\alpha\in \NN^{\oplus E}$ and $a,b\in X$, we set
\begin{align*}
P_{k,\alpha}(X)&=\{\vec{x}\in P_k(X)\mid \vec{\ell}(\vec{x})=\alpha\},\\
P_{k,\alpha}(X)_{a\to\bullet}&=\{\vec{x}\in P_k(X)\mid \vec{\ell}(\vec{x})=\alpha,\ x_0=a\},\\
P_{k,\alpha}(X)_{\bullet\to b}&=\{\vec{x}\in P_k(X)\mid \vec{\ell}(\vec{x})=\alpha,\ x_k=b\},\\
P_{k,\alpha}(X)_{a\to b}&=\{\vec{x}\in P_k(X)\mid \vec{\ell}(\vec{x})=\alpha,\ x_0=a,\ x_k=b\}.
\end{align*}
Define
\[
    \MC_{k,\alpha}(X),\quad
    \MC_{k,\alpha}(X)_{a\to\bullet},\quad
    \MC_{k,\alpha}(X)_{\bullet\to b},\quad
    \MC_{k,\alpha}(X)_{a\to b}
\]
to be the free abelian groups generated by
\[
    P_{k,\alpha}(X),\quad
    P_{k,\alpha}(X)_{a\to\bullet},\quad
    P_{k,\alpha}(X)_{\bullet\to b},\quad
    P_{k,\alpha}(X)_{a\to b},
\]
respectively.
These groups form subcomplexes of $\MC_{\ast,\lvert\alpha\rvert}(X)$.
Write
\[
    \MH_{k,\alpha}(X),\quad
    \MH_{k,\alpha}(X)_{a\to\bullet},\quad
    \MH_{k,\alpha}(X)_{\bullet\to b},\quad
    \MH_{k,\alpha}(X)_{a\to b}
\]
for the $k$-th homology group of these subcomplexes.
Then we have
\[
    \MH_{\ast,\ell}(X)
    \cong
    \bigoplus_{\substack{\alpha\in \NN^{\oplus E}\\ \lvert\alpha\rvert=\ell}}
    \MH_{\ast,\alpha}(X),
\]
and, for every $\alpha\in \NN^{\oplus E}$,
\[
    \MH_{\ast,\alpha}(X)
    \cong \bigoplus_{a\in X}\MH_{\ast,\alpha}(X)_{a\to\bullet}
    \cong \bigoplus_{b\in X}\MH_{\ast,\alpha}(X)_{\bullet\to b}
    \cong \bigoplus_{a,b\in X}\MH_{\ast,\alpha}(X)_{a\to b}.
\]

Let $E'\subseteq E$ be a subset.
A proper $k$-chain $\vec{x}\in P_k(X)$ is called \emph{$E'$-supported} if $\supp(\vec{\ell}(\vec{x}))=E'$.
The $E'$-supported chain group $\MC_{k,\ell}(X;E')\subset\MC_{k,\ell}(X)$ is generated by $E'$-supported proper $k$-chains.
These groups form a subcomplex of $\MC_{\ast,\ell}(X)$, and its homology is denoted by $\MH_{\ast,\ell}(X;E')$.

\subsection{Hyperplane arrangements}
We recall basic facts about hyperplane arrangements which will be needed later.
Standard references are \cite{Orlik1992,Stanley2007}.

Let $\A$ be an affine hyperplane arrangement in $\mathbb R^d$, that is, a locally finite set of affine hyperplanes in $\mathbb R^d$.
Here, $\A$ is said to be locally finite if, for every compact subset $C\subseteq \mathbb R^d$, only finitely many elements of $\A$ intersect $C$.
For each $H\in\A$, fix an affine defining function $f_H:\mathbb R^d\to\mathbb R$ with $H=f_H^{-1}(0)$.
The arrangement $\A$ stratifies $\mathbb R^d$ into relatively open, locally polyhedral faces.
The \emph{sign vector} of a face $F$ is
\[
        (F_H)_{H\in\A}\in\{+,-,0\}^{\A},        
\]
where $F_H=\sgn(f_H(x))$ for any point $x$ in the relative interior of $F$.
We write $\Face(\A)$ for the face poset, ordered by
\[
        F\leq G
        \quad\Longleftrightarrow\quad
        F\subseteq\overline G.
\]
Equivalently, on sign vectors the order is the coordinatewise order generated by
$0<+$ and $0<-$, with $+$ and $-$ incomparable.
Thus chambers are the maximal faces; the set of chambers is denoted by
$\Ch(\A)$.

For two chambers $C,C'\in\Ch(\A)$, let
\[
        S(C,C'):=\{H\in\A\mid C_H\neq C'_H\}
\]
be the set of hyperplanes separating $C$ and $C'$.
The \emph{tope graph} $\Tope(\A)$ is the graph with vertex set $\Ch(\A)$ in which two chambers $C,C'$ are adjacent if $\#S(C,C')=1$.
Its shortest-path metric satisfies
\[
        d(C,C')=\#S(C,C').
\]
Thus the sign vectors define an isometric embedding $\Tope(\A)\hookrightarrow\{+,-\}^{\A}$.
In particular, one can decompose the magnitude homology $\MH_{k,\ell}(\Tope(\A))$ into the groups $\MH_{k,\alpha}(\Tope(\A))$ for $\alpha\in \NN^{\oplus \A}$ with $\lvert\alpha\rvert=\ell$.
For a subarrangement $\B\subseteq\A$, the construction above gives the $\B$-supported magnitude homology $\MH_{k,\ell}(\Tope(\A);\B)$.
Following \cite{Koizumi_Liu}, we define the \emph{interior magnitude homology} of $\Tope(\A)$ by
\[
\MH_{k,\ell}^\circ (\Tope(\A))=\MH_{k,\ell}(\Tope(\A);\A).
\]
Note that we trivially have $\MH_{k,\ell}^\circ (\Tope(\A))=0$ for infinite arrangements.

The \emph{intersection poset} of $\A$ is defined by
\[
L(\A)=\biggl\{\bigcap_{H\in\mathcal{B}}H\mid \mathcal{B}\subseteq \A\biggr\}\setminus\{\varnothing\},
\]
where the order is given by reverse inclusion.
An element of $L(\A)$ is called a \emph{flat}; the atoms of $L(\A)$ are precisely the
hyperplanes of $\A$.
For a flat $X\in L(\A)$, define the \emph{localization} and \emph{restriction}
\[
    \A_X:=\{H\in\A\mid X\subseteq H\},
    \qquad
    \A^X:=\{H\cap X\mid H\in\A\setminus\A_X,\ H\cap X\neq\varnothing\},
\]
where repeated hyperplanes are omitted from $\A^X$, which is regarded as an
arrangement in $X$.  If $\widehat 0=\mathbb R^d$ denotes the minimum of $L(\A)$,
then $L(\A_X)$ identifies with the interval $[\widehat{0},X]$, while $L(\A^X)$
identifies with the principal filter $L(\A)_{\geq X}$.
We say that $\A$ is \emph{central} if $\bigcap_{H\in\A}H\neq\varnothing$.
For a face $F\in\Face(\A)$, we similarly write
\[
    \A_F:=\{H\in\A\mid F\subseteq H\}.
\]
This localization is central.

Let $\A$ be a central hyperplane arrangement.
Then $\#\A$ is finite since $\A$ is locally finite.
We define
\[
        F_\A:=\bigcap_{H\in\A}H
\]
and call it the \emph{center} of $\A$.
We define the \emph{rank} of $\A$ by $\rank(\A):=\codim F_\A$.
Every chamber $C$ has an opposite chamber, denoted by $-C$, obtained by changing all signs.

\begin{example}\label{topeboolean}
    Let $\A=\Bl_d$ be the $d$-th Boolean arrangement, consisting of the $d$ coordinate hyperplanes in $\mathbb R^d$.
    Then $\Tope(\A)$ is the whole hypercube $\{+,-\}^d$.
\end{example}

\begin{example}\label{topegeneric}
    Let $r\geq 2$ and let $\A$ be a central hyperplane arrangement in $\mathbb R^r$ with $\#\A\geq r$.
    We say that $\A$ is \emph{generic} if every $r$-element subarrangement of $\A$ has rank $r$.
    We write $U_{r,n}$ for a generic central hyperplane arrangement of $n$ hyperplanes in $\mathbb R^r$.
    For $\A=U_{r,r+1}$, the tope graph $\Tope(\A)$ is isomorphic to the induced subgraph of $\{+,-\}^{\A}\cong Q_{r+1}$ obtained by removing $(+,\dots,+)$ and $(-,\dots,-)$.
\end{example}

\begin{example}\label{topecoxeter}
    Let $(W,S)$ be a Coxeter system of finite or affine type.
    Let $\A_W$ be the associated Coxeter arrangement.
    Then $W$ acts freely and transitively on the chambers of $\A_W$.
    The chambers of $\A_W$ are in bijection with the elements of $W$, and adjacency of chambers corresponds to multiplication on the right by elements of $S$.
    Hence the tope graph $\Tope(\A_W)$ is isomorphic to the Cayley graph $\Cay(W,S)$ of $W$ with respect to $S$.
    Since $S$ is finite, the Cayley graph $\Cay(W,S)$ has finite degree and hence finite balls; thus it is of finite type.
    Its magnitude weighting is given by the reciprocal of the Poincar\'e series:
    \[
    \wt(b)=\frac{1}{\sum_{w\in W}q^{\ell(w)}}.
    \]
    Here, $\ell(w)$ denotes the length of $w\in W$.
\end{example}

\begin{definition}[Tits product]
    For sign vectors $F,G\in \{+,-,0\}^\A$, define their \emph{Tits product} $F\circ G$ by
    \[
    (F\circ G)_H=\begin{cases}
    F_H, &\text{if }F_H\neq 0;\\
    G_H, &\text{if }F_H=0.
    \end{cases}
    \]
    If $F,G\in\Face(\A)$, then $F\circ G\in \Face(\A)$.
    Geometrically, if $x\in F$ and $y\in G$ are relative interior points, then $F\circ G$ contains $(1-\varepsilon)x+\varepsilon y$ for all sufficiently small $\varepsilon>0$.
    If $F\in \Face(\A)$ and $G\in \Ch(\A)$, then $F\circ G\in \Ch(\A)$.
\end{definition}

For a face $F$ of an affine hyperplane arrangement $\A$, write
\[
        \Ch(\A)_F:=\{C\in\Ch(\A)\mid F\leq C\}.
\]
This is the set of chambers whose closures contain $F$.
An interval $[C,D]_{\Tope(\A)}$ is called \emph{facial} if
$[C,D]_{\Tope(\A)}=\Ch(\A)_F$ for some $F\in\Face(\A)$.
In this situation, the relevant face is the one whose zero coordinates are the
hyperplanes separating $C$ and $D$, and whose nonzero coordinates are the
common signs of $C$ and $D$.

\begin{theorem}[{\cite[Theorem 2.2]{Edelman1985}; see also \cite[Theorem 4.4.2]{Bjorner1999}}]\label{thm:interval_homotopy}
    Let $\A$ be a central hyperplane arrangement and let $C,D\in \Ch(\A)$ be distinct chambers.
    If the interval $[C,D]_{\Tope(\A)}$ is facial, that is,
    \[
    [C,D]_{\Tope(\A)}=\Ch(\A)_F
    \]
    for some $F\in \Face(\A)$, then the open interval $(C,D)_{\Tope(\A)}$ is homotopy equivalent to a sphere $S^{\codim F-2}$.
    Otherwise, the open interval $(C,D)_{\Tope(\A)}$ is contractible.
\end{theorem}

\section{Localization and periodicity}

First we prove a localization theorem used in the proof of the main theorem.
The theorem shows that the $\alpha$-summand $\MH_{k,\alpha}(\Tope(\A))_{\bullet\to B}$ either localizes to $\operatorname{supp}(\alpha)$ when this support is the localization at a face incident to $B$, or vanishes otherwise.

\begin{theorem}[Localization]\label{thm:localization}
Let $\A$ be an affine hyperplane arrangement and let $B\in \Ch(\A)$.
Let $\zero\neq \alpha\in\NN^{\oplus \A}$ and put $\A_{\alpha}=\supp(\alpha)$.
If there exists a face $F\in \Face(\A)_{<B}$ with $\A_F=\A_{\alpha}$, then
\[
    \MH_{k,\alpha}(\Tope(\A))_{\bullet\to B}
    \cong
    \MH_{k,\alpha|_{\A_\alpha}}(\Tope(\A_\alpha))_{\bullet\to B_\alpha},
\]
where $B_\alpha$ is the unique chamber of $\A_\alpha$ containing $B$.
Otherwise, we have
\[
    \MH_{k,\alpha}(\Tope(\A))_{\bullet\to B}=0.
\]
\end{theorem}

\begin{proof}
Let $U$ be the unique chamber of $\A\setminus\A_\alpha$ containing $B$.
Consider a proper $k$-chain
\[
\vec{C}=(C_0,\ldots,C_k=B)\in P_{k,\alpha}(\Tope(\A))_{\bullet\to B}.
\]
Since $\vec{C}$ crosses no hyperplane in $\A\setminus\A_\alpha$, we have $C_0,\dots,C_k\subseteq U$.
Let $\Ch(\A)_{\subseteq U}$ denote the set of chambers of $\A$ contained in $U$, and let $\Tope(\A)_{\subseteq U}$ denote the induced subgraph of $\Tope(\A)$ spanned by $\Ch(\A)_{\subseteq U}$.
The subgraph $\Tope(\A)_{\subseteq U}$ is convex, hence isometric, in
$\Tope(\A)$.
Thus
\[
        \MC_{k,\alpha}(\Tope(\A))_{\bullet\to B}
        =
        \MC_{k,\alpha}(\Tope(\A)_{\subseteq U})_{\bullet\to B}.
\]

Let $\bigcap \A_\alpha$ denote the intersection of all hyperplanes of $\A_\alpha$.
Suppose first that $F=U\cap \bigcap \A_\alpha\neq\varnothing$; see \Cref{fig:localization}.
Then $F$ is the unique element of $\Face(\A)_{<B}$ with $\A_F=\A_\alpha$.
There is a bijection
\[
        \iota_F:\Ch(\A)_{\subseteq U}\xrightarrow{\sim}\Ch(\A_\alpha)
\]
which sends $C\in \Ch(\A)_{\subseteq U}$ to the unique chamber of $\A_\alpha$ containing $C$.
Moreover, $\iota_F$ is an isometry.
Thus $\iota_F$ identifies $\MH_{k,\alpha}(\Tope(\A))_{\bullet\to B}=\MH_{k,\alpha}(\Tope(\A)_{\subseteq U})_{\bullet\to B}$ with $\MH_{k,\alpha|_{\A_\alpha}}(\Tope(\A_\alpha))_{\bullet\to B_\alpha}$.

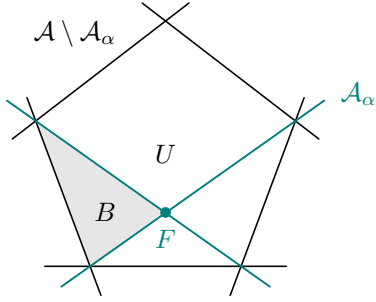
\begin{figure}[ht]
\begin{tikzpicture}[
    scale=0.8,
    line cap=round,
    line join=round,
    blackline/.style={black, line width=0.6pt},
    redline/.style={teal, line width=0.6pt}
]
% 特別な交点を先に指定する
\coordinate (T)  at (0,3.7);       % 上側の二重点
\coordinate (L)  at (-2.15,2.05); % 左上の三重点
\coordinate (R)  at ( 2.15,2.05); % 右上の三重点
\coordinate (BL) at (-1.25,-0.35);% 左下の三重点
\coordinate (BR) at ( 1.25,-0.35);% 右下の三重点
\coordinate (F) at ( 0,0.54);% 交点
\filldraw[black!10] (F) -- (L) -- (BL) -- cycle;
% 黒い直線配置
\draw[blackline]($(L)!-0.18!(T)$) -- ($(L)!1.18!(T)$);
\draw[blackline]($(R)!-0.18!(T)$) -- ($(R)!1.18!(T)$);
\draw[blackline]($(L)!-0.16!(BL)$) -- ($(L)!1.20!(BL)$);
\draw[blackline]($(R)!-0.16!(BR)$) -- ($(R)!1.20!(BR)$);
\draw[blackline]($(BL)!-0.30!(BR)$) -- ($(BL)!1.30!(BR)$);
% 赤い2直線
\draw[redline]($(L)!-0.14!(BR)$) -- ($(L)!1.14!(BR)$);
\draw[redline]($(BL)!-0.14!(R)$) -- ($(BL)!1.14!(R)$);
\draw[redline]($(L)!-0.14!(BR)$) -- ($(L)!1.14!(BR)$);
\draw[redline]($(BL)!-0.14!(R)$) -- ($(BL)!1.14!(R)$);
\fill[teal] (F) circle[radius=2.5pt];
% ラベル
\node at (-1.5,3.5) {$\mathcal A\setminus\mathcal A_{\alpha}$};
\node[teal] at (3.2,2.5) {$\mathcal A_{\alpha}$};
\node at (0,1.5) {$U$};
\node at (-1,0.55) {$B$};
\node[teal] at (0,0.1) {$F$};
\end{tikzpicture}
\caption{The first case in the proof of \Cref{thm:localization}}
\label{fig:localization}
\end{figure}

Suppose next that $U\cap \bigcap \A_\alpha=\varnothing$.
In this case, there is no face $F\in \Face(\A)_{<B}$ with $\A_F=\A_\alpha$.
We want to show that $\MH_{k,\alpha}(\Tope(\A))_{\bullet\to B}=0$.
The sign vectors of faces contained in $U$, restricted to the coordinates in
$\A_\alpha$, form a realizable conditional oriented matroid (COM) \cite{Bandelt2018}.
In particular, $\Tope(\A)_{\subseteq U}$ is the tope graph of a realizable COM\@.
As in the proof of \cite[Proposition 6.10]{Koizumi_Liu}, its antipodal subgraphs
supported on all of $\A_\alpha$ correspond to faces whose realization is contained in
$\bigcap\A_\alpha$ and meets $U$.  The assumption
$U\cap\bigcap\A_\alpha=\varnothing$ therefore excludes such a subgraph.  The
vanishing theorem \cite[Theorem 5.3]{Koizumi_Liu} now gives
\[
\MH_{k,\lvert\alpha\rvert}(\Tope(\A)_{\subseteq U};\A_\alpha)=0.
\]
Since the group $\MH_{k,\alpha}(\Tope(\A)_{\subseteq U})_{\bullet\to B}$ is a direct summand of $\MH_{k,\lvert\alpha\rvert}(\Tope(\A)_{\subseteq U};\A_\alpha)$, it also vanishes.
\end{proof}

Next we state a periodicity theorem for central hyperplane arrangements.
Computations in \cite{Koizumi_Liu} suggest a periodic pattern in the magnitude Betti numbers under the shift $(k,\ell)\mapsto(k+r,\ell+n)$, where
$r=\rank\A$ and $n=\#\A$.
This pattern does not give an isomorphism for each total-length summand.
After refining total length to the length vector, it becomes the following isomorphism.

\begin{theorem}[Periodicity]\label{thm:periodicity}
Let $\A$ be a central hyperplane arrangement of rank $r$, and fix a chamber $B\in \Ch(\A)$.
Let $\alpha\in\NN^{\oplus \A}$ satisfy $\supp(\alpha)=\A$ (equivalently, $\alpha\geq\one$).
Then
\[
    \MH_{k,\alpha}(\Tope(\A))_{\bullet\to B}
    \cong
    \MH_{k-r,\alpha-\one}(\Tope(\A))_{\bullet\to B},
\]
where the right-hand side is interpreted as zero for $k<r$.
\end{theorem}

The next two sections are devoted to the proof of \Cref{thm:periodicity}.

\section{Poset-topological lemmas}

The proof of the periodicity theorem reduces to the acyclicity of certain finite posets.
We record three elementary tools: an acyclic extension lemma, a gluing lemma for
lower ideals, and Alexander duality for PL balls.

\begin{lemma}[Acyclic extension lemma]\label{lem:homological-quillen}
Let $i\colon P\hookrightarrow Q$ be an inclusion of posets.  Suppose that for every
$x\in Q\setminus P$, the lower fiber
\[
        P_{\leq x}=\{p\in P\mid p\leq x\}
\]
is acyclic.  Then the inclusion $i\colon P\hookrightarrow Q$ induces isomorphisms on
reduced homology groups.
\end{lemma}

\begin{proof}
Put $K=\Delta Q$ and $L=\Delta P$.  For a nonempty simplex
$\sigma=(q_0<\cdots<q_m)$ of $K$, set $m(\sigma)=q_m$ and define
\[
        \Phi(\sigma):=\Delta P_{\leq m(\sigma)}\subseteq L.
\]
This complex is acyclic.  Indeed, if $m(\sigma)\in P$, then
$P_{\leq m(\sigma)}$ has the maximum element $m(\sigma)$; if
$m(\sigma)\notin P$, this is exactly the hypothesis.  Moreover, if
$\tau\leq\sigma$ is a face, then $m(\tau)\leq m(\sigma)$, and hence
$\Phi(\tau)\subseteq\Phi(\sigma)$.  Thus $\Phi$ is an acyclic carrier from
$K$ to $L$.
By the augmented acyclic carrier theorem
\cite[Theorem~13.3]{Munkres},
there exists an
augmentation-preserving chain map
\[
        r:C_\ast(K)\longrightarrow C_\ast(L)
\]
carried by $\Phi$.  The maps $r\circ i_\#$ and
$\mathrm{id}_{C_\ast(L)}$ are carried by the same acyclic carrier on $L$, so they
are chain homotopic.  Similarly, the identity on $C_\ast(K)$ and $i_\#\circ r$
are both carried by the carrier
\[
        \Psi(\sigma)=\Delta Q_{\leq m(\sigma)},
\]
whose values are cones.  Hence $i_\#\circ r\simeq\mathrm{id}_{C_\ast(K)}$.
The homotopies preserve augmentations, and therefore $i_\ast$ is an isomorphism
on reduced homology.
\end{proof}

\begin{lemma}[Gluing lemma]\label{lem:gluing}
Let $P$ be a poset, and let $\Sigma$ be a nonempty finite family of lower
ideals of $P$.  Assume that $\Sigma$ is closed under pairwise intersections.  If every
$Q\in\Sigma$ is acyclic, then
$\bigcup_{Q\in\Sigma}Q$
is acyclic.
\end{lemma}

\begin{proof}
Write $\Sigma=\{Q_1,\ldots,Q_m\}$ and argue by induction on $m$.
The case $m=1$ is clear.
Let $m\geq 2$.
We may assume that $Q_m$ is maximal among $Q_1,\dots,Q_m$.
Put $I=Q_1\cup\cdots\cup Q_{m-1}$.
By induction, $I$ is acyclic.  Also
\[
        I\cap Q_m=\bigcup_{i=1}^{m-1}(Q_i\cap Q_m),
\]
and the family $\{Q_i\cap Q_m\}_{i=1}^{m-1}$ is again closed under
pairwise intersections.  Thus $I\cap Q_m$ is acyclic by induction.
Since $I$ and $Q_m$ are lower ideals, every chain in $I\cup Q_m$ is contained
in either $I$ or $Q_m$.  Hence
\[
        \Delta(I\cup Q_m)=\Delta I\cup\Delta Q_m,
        \quad
        \Delta I\cap\Delta Q_m=\Delta(I\cap Q_m).
\]
The reduced Mayer--Vietoris sequence now gives the claim.
\end{proof}

\begin{lemma}[Alexander duality for PL balls]
\label{lem:alexander_ball}
Let $\Delta$ be a finite simplicial complex whose realization is a
$d$-dimensional PL ball.  Let
$V(\Delta)=X\sqcup Y$
be a partition of its vertex set, and let $\Delta_X$ and $\Delta_Y$ be the
induced subcomplexes on $X$ and $Y$.  Then, for every integer $k$, there is an isomorphism
\[
        H_k(\Delta_X,\Delta_X\cap\partial\Delta)
        \cong
        \widetilde H^{\,d-k-1}(\Delta_Y).
\]
\end{lemma}

\begin{proof}
If $Y=\varnothing$, then $\Delta_X=\Delta$ and
\[
H_k(\Delta,\partial\Delta)
\cong
\begin{cases}
\ZZ,&k=d;\\
0,&k\neq d,
\end{cases}
\]
while $\widetilde H^{d-k-1}(\varnothing)$ has the same value under the convention
$\widetilde H^{-1}(\varnothing)=\ZZ$.  Thus the claim holds in this case.
Assume henceforth that $Y\neq\varnothing$.
Let $v$ be a new vertex and form
\[
        \Sigma:=\Delta\cup_{\partial\Delta}(v\ast\partial\Delta).
\]
Then $\Sigma$ is a triangulated $d$-sphere.  Put $U=X\cup\{v\}$.  The induced
subcomplex of $\Sigma$ on $Y$ is $\Delta_Y$, and the induced subcomplex on
$U$ is
\[
        \Sigma_U=
        \Delta_X\cup_{\Delta_X\cap\partial\Delta}
        \bigl(v\ast(\Delta_X\cap\partial\Delta)\bigr).
\]
Thus $\Sigma_U$ is obtained from $\Delta_X$ by coning off
$\Delta_X\cap\partial\Delta$, and therefore
\[
        \widetilde H_k(\Sigma_U)
        \cong
        H_k(\Delta_X,\Delta_X\cap\partial\Delta).
\]

Since $\Sigma_U$ is the induced subcomplex on the complementary vertex set $U=V(\Sigma)\setminus Y$, the usual deformation retraction
$|\Sigma|\setminus |\Delta_Y|\simeq |\Sigma_U|$ applies.  Alexander duality in
the sphere $|\Sigma|$ \cite[Corollary~3.45]{Hatcher} gives
\[
        \widetilde H_k(|\Sigma|\setminus |\Delta_Y|)
        \cong
        \widetilde H^{\,d-k-1}(\Delta_Y).
\]
Combining the preceding isomorphisms proves the lemma.
\end{proof}

\section{Proof of the periodicity theorem}

We now prove \Cref{thm:periodicity}.
We first decompose the magnitude chain complex with length vector $\alpha$ into two parts: the special part and the non-special part.  
The homology of the special quotient is isomorphic to the homology with length vector $\alpha-\one$, shifted by $r$.
The difficult step is to show that the non-special part is acyclic.  
This part is handled using the poset-topological lemmas proved in the previous section together with an inductive argument.

\subsection{Special chains}
Fix a central hyperplane arrangement $\A=\{H_1,\dots,H_n\}$ with $\rank\A=r$.
We also fix a chamber $B\in \Ch(\A)$ and a vector $\alpha\in\NN^{\oplus n}$ with $\alpha\geq\one$.
If $\A=\varnothing$, then $r=0$, $\alpha=\zero$, and the periodicity theorem is the
identity isomorphism.  We therefore assume throughout this section that $n>0$.

\begin{definition}
    Let $\vec C=(C_0,\ldots,C_k)\in P_{k,\alpha}(\Tope(\A))$.
    Put
    \[
            \lambda_i:=\vec{\ell}(C_0,\ldots,C_i),
            \quad 0\leq i\leq k.
    \]
    The chain $\vec C$ is called \emph{special} if $\lambda_i=\one$ for some
    $i$.
    Otherwise it is called \emph{non-special}.
    Let $N_{k,\alpha}(\Tope(\A))$ be the subgroup of $\MC_{k,\alpha}(\Tope(\A))$ generated by non-special chains.
    Define $N_{k,\alpha}(\Tope(\A))_{A\to \bullet}$, $N_{k,\alpha}(\Tope(\A))_{\bullet\to B}$, and $N_{k,\alpha}(\Tope(\A))_{A\to B}$ in a similar manner.
\end{definition}

Since the vectors $\lambda_i$ are strictly increasing in the coordinatewise
order, the index $i$ such that $\lambda_i=\one$, when it exists, is unique.

\begin{lemma}\label{lem:nonspecial_subcomplex}
$N_{\ast,\alpha}(\Tope(\A))$ is a subcomplex of $\MC_{\ast,\alpha}(\Tope(\A))$.
\end{lemma}

\begin{proof}
Let $\vec C=(C_0,\ldots,C_k)$ be non-special, and suppose that
\[
    d(C_{i-1},C_{i+1})
    =d(C_{i-1},C_i)+d(C_i,C_{i+1}).
\]
This means that the two sets of hyperplanes crossed in the two adjacent steps
are disjoint.
Thus
\[
    \vec{d}(C_{i-1},C_{i+1})
    =\vec{d}(C_{i-1},C_i)+\vec{d}(C_i,C_{i+1}).
\]
Hence deleting $C_i$ preserves the total length vector $\alpha$.
It cannot create a prefix with cumulative length vector $\one$, because every prefix in the shortened chain is already a prefix in the original chain.
Therefore all nonzero boundary terms of a non-special chain are non-special.
\end{proof}

\begin{definition}
Define the quotient complex
\[
        Sp_{\ast,\alpha}(\Tope(\A)):=\MC_{\ast,\alpha}(\Tope(\A))/N_{\ast,\alpha}(\Tope(\A)).
\]
We also define $Sp_{\ast,\alpha}(\Tope(\A))_{A\to \bullet}$, $Sp_{\ast,\alpha}(\Tope(\A))_{\bullet\to B}$, and $Sp_{\ast,\alpha}(\Tope(\A))_{A\to B}$ in a similar manner.
\end{definition}

\begin{proposition}\label{prop:tensor_decomp}
There is a natural isomorphism of chain complexes
\[
        Sp_{\ast,\alpha}(\Tope(\A))_{A\to B}
        \cong
        \MC_{\ast,n}(\Tope(\A))_{A\to -A}
        \otimes
        \MC_{\ast,\alpha-\one}(\Tope(\A))_{-A\to B}.
\]
\end{proposition}

\begin{proof}
A special chain $\vec C=(A=C_0,\ldots,C_k=B)$ has a unique index $j$ such that
\[
        \vec{\ell}(C_0,\ldots,C_j)=\one.
\]
Then $C_j=-A$, the first part $(A=C_0,\ldots,C_j=-A)$ is a length-$n$ chain from $A$ to $-A$, and the second part $(-A=C_j,\ldots,C_k=B)$ is a chain from $-A$ to $B$ with length vector $\alpha-\one$.
Conversely, concatenating such two chains gives a special
chain with length vector $\alpha$.
Note that every length-$n$ chain from $A$ to $-A$ is geodesic and has length vector $\one$.

In the quotient by non-special chains, the boundary term deleting the gluing
vertex $C_j$ is zero: it is either zero already, or it is non-special.  Thus
\[
\begin{aligned}
        \partial\vec C
        &\equiv
        \sum_{i=1}^{j-1}(-1)^i\partial_i\vec C
        +(-1)^j\sum_{i=1}^{k-j-1}(-1)^i\partial_{i+j}\vec C
        \quad\text{mod }N_{\ast,\alpha}(\Tope(\A)).
\end{aligned}
\]
This is exactly the tensor-product differential with the first tensor factor
in degree $j$.
\end{proof}

\begin{proposition}\label{prop:special_homology}
The homology groups of $Sp_{\ast,\alpha}(\Tope(\A))_{\bullet\to B}$ are given by
\[
        H_kSp_{\ast,\alpha}(\Tope(\A))_{\bullet\to B}
        \cong
        \MH_{k-r,\alpha-\one}(\Tope(\A))_{\bullet\to B},
\]
where the right-hand side is interpreted as zero for $k<r$.
\end{proposition}

\begin{proof}
For each chamber $A\in \Ch(\A)$, \Cref{intervalMH} identifies
$\MC_{\ast,n}(\Tope(\A))_{A\to -A}$ with a two-fold shift of the chain complex of the open
interval $(A,-A)_{\Tope(\A)}$ in the tope graph.
Since $A$ and $-A$ differ on every hyperplane, the closed interval $[A,-A]_{\Tope(\A)}$ is the facial
interval
\[
        [A,-A]_{\Tope(\A)}=\Ch(\A)_{F_\A}.
\]
By the Edelman--Walker theorem (\Cref{thm:interval_homotopy}), the open
interval $(A,-A)_{\Tope(\A)}$ has the homology of a sphere of dimension $r-2$.
When $r=1$, this means that the open interval is empty and is interpreted as
$S^{-1}$, with reduced homology $\widetilde H_{-1}(\varnothing)=\ZZ$.
Consequently
\[
    \MH_{\ast,n}(\Tope(\A))_{A\to -A}\cong \ZZ[r],
\]
where $\ZZ[r]$ denotes a graded abelian group with $\ZZ$ concentrated in degree
$r$.  Since all chain groups are free abelian, the K\"unneth theorem applied to
\Cref{prop:tensor_decomp} gives
\[
        H_kSp_{\ast,\alpha}(\Tope(\A))_{A\to B}
        \cong \MH_{k-r,\alpha-\one}(\Tope(\A))_{-A\to B}.
\]
The claim now follows by taking the direct sum over $A\in \Ch(\A)$.
\end{proof}

Therefore \Cref{thm:periodicity} follows once we prove that
$N_{\ast,\alpha}(\Tope(\A))$ is acyclic for every $\alpha\geq\one$.

\subsection{Reduction to a poset statement}\label{sec:poset}
We now translate non-special chains into relative chains in a subposet of $\NN^{\oplus n}$.
Fix a base chamber $A\in\Ch(\A)$.  For each hyperplane $H_i$, choose the side
containing $A$ as the positive side.  This identifies the chambers with a
subset
\[
        \Cset_A\subseteq\{0,1\}^n,
\]
where $A$ corresponds to $\zero\in \Cset_A$.  Since $\A$ is central, the
opposite chamber $-A$ corresponds to $\one\in \Cset_A$.

\begin{definition}[The posets $\Tpos_A$ and $\Npos_A$]
Define two subposets of $\NN^{\oplus n}$ by
\[
        \Tpos_A
        :=\{\beta\in\NN^{\oplus n}\mid \ol\beta\in\Cset_A\}\setminus\{\one\},
        \quad
        \Npos_A
        :=\{\beta\in\NN^{\oplus n}\mid \ol\beta\notin\Cset_A\}\cup\{\one\}.
\]
Thus $\NN^{\oplus n}=\Tpos_A\sqcup\Npos_A$.
When the base chamber $A$ is clear, we write simply $\Tpos$ and $\Npos$.
\end{definition}

For $\alpha\in\NN^{\oplus n}$, let
\[
        \Delta_\alpha,
        \quad
        \Delta_\alpha^\Tpos,
        \quad
        \Delta_\alpha^\Npos
\]
be the order complexes of $\NN^{\oplus n}_{\leq\alpha}$,
$\Tpos_{A,\leq\alpha}$, and $\Npos_{A,\leq\alpha}$, respectively.
The order complex $\Delta_\alpha$ is the standard staircase triangulation of
$\Delta^{\alpha_1}\times\cdots\times\Delta^{\alpha_n}$,
and is therefore a PL ball of dimension $\lvert\alpha\rvert$.

\begin{lemma}\label{lem:relative_generators}
A simplex
\[
        \beta^0<\beta^1<\cdots<\beta^k
\]
of $\Delta_\alpha$ is not contained in $\partial\Delta_\alpha$ if and only if
for every coordinate $i$, the set
\[
        \{\beta_i^0,\beta_i^1,\ldots,\beta_i^k\}
\]
is the full set $\{0,1,\ldots,\alpha_i\}$.
Equivalently, the chain starts at $\zero$, ends at $\alpha$, and each
successive difference $\beta^j-\beta^{j-1}$ belongs to $\{0,1\}^n\setminus\{\zero\}$.
\end{lemma}

\begin{proof}
Under the identification of $\Delta_\alpha$ with the staircase triangulation of
$\Delta^{\alpha_1}\times\cdots\times\Delta^{\alpha_n}$, the $i$-th projection
records the vertices of the simplex $\Delta^{\alpha_i}$ that occur in the
$i$-th coordinate.  A simplex is in the boundary of the product exactly when,
for some $i$, this projection lies in the boundary of $\Delta^{\alpha_i}$;
that is, when at least one value in $\{0,1,\ldots,\alpha_i\}$ is missing.
This proves the first assertion.  The equivalent formulation follows from the
monotonicity of the chain.
\end{proof}

\begin{lemma}\label{lem:nonspecial_to_poset}
For every $\alpha\in\NN^{\oplus n}$ with $\alpha\geq\one$, there is an isomorphism of chain complexes
\[
        N_{\ast,\alpha}(\Tope(\A))_{A\to \bullet}
        \cong
        C_\ast(\Delta_\alpha^\Tpos,\Delta_\alpha^\Tpos\cap\partial\Delta_\alpha).
\]
\end{lemma}

\begin{proof}
Let $\vec C=(A=C_0,\ldots,C_k)\in P_{k,\alpha}(\Tope(\A))_{A\to \bullet}$ be a non-special chain.
Define
\[
        \lambda_j:=\vec{\ell}(C_0,\ldots,C_j),
        \quad 0\leq j\leq k.
\]
Then $\zero=\lambda_0<\lambda_1<\cdots<\lambda_k=\alpha$,
and each difference $\lambda_j-\lambda_{j-1}$ is a nonzero vector in $\{0,1\}^n$.
The parity vector $\ol{\lambda_j}$ is the sign vector of the chamber $C_j$ relative to $A$.
Since the chain is non-special, no $\lambda_j$ is equal to $\one$.  Hence all vertices $\lambda_j$ lie in $\Tpos_A$.

Moreover, in each coordinate $i$, the value of $\lambda_{j,i}$ increases by
one exactly when the step $C_{j-1}\to C_j$ crosses $H_i$.  Since the total
number of crossings of $H_i$ is $\alpha_i$, the chain
$\lambda_0<\cdots<\lambda_k$ uses every value $0,1,\ldots,\alpha_i$ in the
$i$-th coordinate.  By \Cref{lem:relative_generators}, it represents a
relative simplex of
$(\Delta_\alpha^\Tpos,\Delta_\alpha^\Tpos\cap\partial\Delta_\alpha)$.
We send $\vec C$ to $[\lambda_0<\cdots<\lambda_k]$.

Conversely, let
$\lambda_0<\cdots<\lambda_k$ be a simplex of $\Delta_\alpha^\Tpos$ which is
not contained in $\partial\Delta_\alpha$.  By
\Cref{lem:relative_generators}, $\lambda_0=\zero$,
$\lambda_k=\alpha$, and every difference
$\lambda_j-\lambda_{j-1}$ lies in $\{0,1\}^n\setminus\{\zero\}$.  Since
$\ol{\lambda_j}\in\Cset_A$, there is a unique chamber $C_j$ with this sign
vector relative to $A$.  The step from $C_{j-1}$ to $C_j$ crosses precisely
the hyperplanes in the support of $\lambda_j-\lambda_{j-1}$.  Thus
$(C_0,\ldots,C_k)$ is a proper magnitude chain with length vector
$\alpha$.  Since no vertex $\lambda_j$ is $\one$, the chain is non-special.

It remains to check the differentials.  The simplicial boundary also has the two
endpoint terms obtained by deleting $\lambda_0=\zero$ or $\lambda_k=\alpha$.
Deleting $\lambda_0$ misses the value $0$ in every coordinate crossed in the
first step, and deleting $\lambda_k$ misses the value $\alpha_i$ in every
coordinate crossed in the last step.  Hence both endpoint faces lie in
$\partial\Delta_\alpha$ by \Cref{lem:relative_generators}, and they vanish
in the relative chain complex.

For an internal vertex, deleting $C_j$ corresponds to deleting $\lambda_j$.
If $C_j$ is smooth, the adjacent crossing supports are disjoint, so deleting
$\lambda_j$ again gives a relative simplex corresponding to the magnitude
boundary term.  If $C_j$ is not smooth, then the metric equality required in the
magnitude differential fails.  Equivalently, some coordinate jumps by at least
two after deleting $\lambda_j$; by \Cref{lem:relative_generators}, the
resulting simplex lies in $\partial\Delta_\alpha$ and is zero in the relative
chain complex.
Hence the two differentials agree.
\end{proof}

\begin{lemma}\label{lem:T_N_alexander}
For every $\alpha\in\NN^{\oplus n}$ and every $k$, there is an isomorphism
\[
        H_k(\Delta_\alpha^\Tpos,
             \Delta_\alpha^\Tpos\cap\partial\Delta_\alpha)
        \cong
        \widetilde H^{\,|\alpha|-k-1}(\Delta_\alpha^\Npos).
\]
\end{lemma}

\begin{proof}
Apply \Cref{lem:alexander_ball} to the PL ball
$\Delta_\alpha$ and to the vertex partition
$\NN^{\oplus n}_{\leq\alpha}
=\Tpos_{A,\leq\alpha}
\sqcup
\Npos_{A,\leq\alpha}$.
\end{proof}

Thus the acyclicity of $N_{\ast,\alpha}(\Tope(\A))$ is reduced to the following purely poset-theoretic statement.

\begin{proposition}\label{prop:target}
For every $A\in \Ch(\A)$ and every $\alpha\in\NN^{\oplus n}$ with $\alpha\geq\one$,
the poset $\Npos_{A,\leq\alpha}$ is acyclic.
\end{proposition}

\subsection{The acyclicity theorem for $\Npos$}
It remains to show that the poset $\Npos_{A,\leq\alpha}$ has no reduced homology.
In this subsection the base chamber $A$ is fixed, and we suppress it from the
notation.

\begin{definition}
Let $a\in\Cset_A\subseteq\{0,1\}^n$ be a chamber sign vector and $C_a$ be the corresponding chamber.
We call $\Tpos_{\leq a}$ facial if the corresponding interval
$[A,C_a]_{\Tope(\A)}$ in $\Tope(\A)$ is facial.
\end{definition}

\begin{lemma}\label{lem:nonfacial_acyclic}
Let $a\in\Tpos_{\leq\one}=\Cset_A\setminus\{\one\}$.  If the interval
$\Tpos_{\leq a}$ is not facial, then $\Npos_{\leq a}$ is acyclic.
\end{lemma}

\begin{proof}
Every relative simplex of
$(\Delta_a^\Tpos,\Delta_a^\Tpos\cap\partial\Delta_a)$ contains both endpoints
$\zero$ and $a$.  For a relative $k$-simplex
\[
[\zero<\lambda_1<\cdots<\lambda_{k-1}<a],
\]
deleting the two endpoints and multiplying by $(-1)^k$ gives a chain isomorphism
with the two-fold degree shift of the augmented simplicial chain complex of the open
interval $\Tpos_{\leq a}\setminus\{\zero,a\}$.  The factor $(-1)^k$ accounts for
our standard simplicial boundary convention.  Thus
\[
        H_\ast(\Delta_a^\Tpos,
             \Delta_a^\Tpos\cap\partial\Delta_a)
        \cong
        \widetilde H_{\ast-2}(\Tpos_{\leq a}\setminus\{\zero,a\}).
\]
By the Edelman--Walker theorem for central hyperplane arrangements (\Cref{thm:interval_homotopy}), the open interval is contractible when
$\Tpos_{\leq a}$ is not facial.  Hence the relative homology group above
vanishes in every degree.  \Cref{lem:T_N_alexander} gives $\widetilde H^\ast(\Delta_a^\Npos)=0$.
By the universal coefficient theorem, all reduced homology groups of
$\Delta_a^\Npos$ also vanish.  Therefore $\Npos_{\leq a}$ is acyclic.
\end{proof}

\begin{lemma}\label{lem:nonfacial_criterion}
Let $a,b\in\Cset_A$ with $a<\one$.  If the coordinatewise meet
$a\wedge b$ is not a chamber sign vector, then $\Npos_{\leq a}$ is acyclic.
\end{lemma}

\begin{proof}
By \Cref{lem:nonfacial_acyclic}, it suffices to show that $\Tpos_{\leq a}$ is not facial.
Assume, for a contradiction, that $\Tpos_{\leq a}$ is facial.  Then there is
a face $F$ such that
\[
        \Tpos_{\leq a}=\{C\in\Ch(\A)\mid F\leq C\},
\]
where the chambers are written in sign coordinates based at $A$.
Let
\[
        S=\{i\mid a_i=1\}.
\]
For each $i\in S$, the chambers $\zero$ and $a$ lie on opposite sides of
$H_i$, and both contain $F$ in their closures.  Hence $F\subseteq H_i$.
For $i\notin S$, every chamber in $\Tpos_{\leq a}$ lies on the same side of
$H_i$ as $A$; hence $F$ is not contained in $H_i$ and lies on the $A$-side of
$H_i$.

Consider the Tits product $F\circ b$.  In the $0$-$1$ coordinates based at
$A$, the coordinates in which $F\subseteq H_i$ inherit the side of $b$, while the
remaining coordinates keep the $A$-side determined by $F$.  Thus
\[
        (F\circ b)_i=
        \begin{cases}
        b_i,& i\in S,\\
        0,& i\notin S,
        \end{cases}
\]
which is exactly $a\wedge b$.  The Tits product of a face and a chamber is a
chamber.  Hence $a\wedge b$ is a chamber sign vector, contradicting the
hypothesis.
\end{proof}

For $\gamma\in\NN^{\oplus n}$ with $\gamma\not\leq\one$, define
\[
        \Sigma(\gamma)
        :=
        \{\gamma-\one_X
        \mid
        \varnothing\neq X\subseteq\{i\mid \gamma_i\geq2\}\},
\]
where $\one_X$ is the characteristic vector of $X$.
The family of lower ideals
$\{\Npos_{\leq\gamma'}\mid \gamma'\in\Sigma(\gamma)\}$ is closed under pairwise intersections, since
\[
        (\gamma-\one_X)\wedge(\gamma-\one_Y)
        =\gamma-\one_{X\cup Y}.
\]

\begin{proposition}[Acyclicity]\label{prop:strong_acyclicity}
Let $q\in\Cset_A$ be a chamber sign vector, and let $\gamma\in\NN^{\oplus n}$.  If
\[
        \gamma\wedge q\in\Npos,
\]
then $\Npos_{\leq\gamma}$ is acyclic.
\end{proposition}

\begin{proof}
We prove the statement by induction on $|\gamma|$.
If $\gamma\in\Npos$, then $\Npos_{\leq\gamma}$ has the maximum element
$\gamma$, and is therefore contractible.  Hence we may assume
$\gamma\in\Tpos$.

Suppose first that $\gamma\leq\one$.  Then $\gamma$ is a chamber sign vector
and $\gamma\neq\one$.
In particular, $\gamma\wedge q\in \Npos\setminus\{\one\}$ and hence $\gamma\wedge q$ is not a chamber sign vector.
\Cref{lem:nonfacial_criterion}
shows that $\Npos_{\leq\gamma}$ is acyclic.

It remains to treat the case $\gamma\not\leq\one$.  Put
$p:=\gamma\wedge q\in\Npos$.
For every $\gamma'\in\Sigma(\gamma)$, we have
\[
        \gamma'\wedge q=p.
\]
Indeed, decreasing a coordinate of $\gamma$ that is at least $2$ does not
change whether that coordinate is positive after taking the meet with the
$0$--$1$ vector $q$.  Since $|\gamma'|<|\gamma|$, the induction hypothesis gives that
$\Npos_{\leq\gamma'}$ is acyclic for every $\gamma'\in\Sigma(\gamma)$.

Define a subposet $W\subseteq\Npos_{\leq \gamma}$ by
\[
        W:=\bigcup_{\gamma'\in\Sigma(\gamma)}\Npos_{\leq\gamma'}.
\]
By \Cref{lem:gluing}, $W$ is acyclic.  We will show that the inclusion
$W\hookrightarrow\Npos_{\leq\gamma}$ induces an isomorphism on reduced
homology, using \Cref{lem:homological-quillen}.
Let $\beta\in\Npos_{\leq\gamma}\setminus W$.
Set
\[
        H:=\{i\mid \gamma_i\geq2\}.
\]
Since $\beta\notin W$, one must have $\beta_i=\gamma_i$ for $i\in H$;
otherwise $\beta\leq\gamma-\one_{\{i\}}$ for some $i\in H$, and hence
$\beta\in W$.  Since $H$ is nonempty, $\beta\neq\one$.  As
$\beta\in\Npos$, it follows that the parity vector $\ol\beta$ is not a
chamber sign vector.

We claim that
\[
        W_{\leq\beta}
        =
        \bigcup_{\varnothing\neq X\subseteq H}
        \Npos_{\leq\beta-\one_X}.
\]
The inclusion from right to left is clear.  Conversely, if
$\eta\in W_{\leq\beta}$, then $\eta\leq\gamma-\one_Y$ for some nonempty
$Y\subseteq H$.  Hence $\eta_i\leq\beta_i-1$ for every $i\in Y$, so
$\eta\leq\beta-\one_Y$.

For each nonempty $X\subseteq H$, put
\[
        \eta_X:=\beta-\one_X.
\]
Because $\gamma\in\Tpos$, its parity vector $\ol\gamma$ is a chamber sign
vector.  Moreover
\[
        \eta_X\wedge\ol\gamma=\ol\beta.
\]
Indeed, for $i\in H$, the equality follows from $\eta_{X,i}\geq 1$ and $\overline{\beta}_i=\overline{\gamma}_i$.
For $i\notin H$, the equality is immediate from $\eta_{X,i}=\beta_i=\overline{\beta}_i\leq \overline{\gamma}_i$.
Since $\ol\beta\in\Npos$ and
$\lvert\eta_X\rvert<\lvert\gamma\rvert$, the induction hypothesis, applied with the chamber sign
vector $\ol\gamma$, gives that $\Npos_{\leq\eta_X}$ is acyclic.

The family
$\{\Npos_{\leq\beta-\one_X}\}_{\varnothing\neq X\subseteq H}$ is closed under
intersections, because
\[
(\beta-\one_X)\wedge(\beta-\one_Y)=\beta-\one_{X\cup Y}.
\]
Hence \Cref{lem:gluing} shows that $W_{\leq\beta}$ is
acyclic.  By \Cref{lem:homological-quillen}, the inclusion
$W\hookrightarrow\Npos_{\leq\gamma}$ induces isomorphisms on reduced homology.
Since $W$ is acyclic, so is $\Npos_{\leq\gamma}$.
\end{proof}

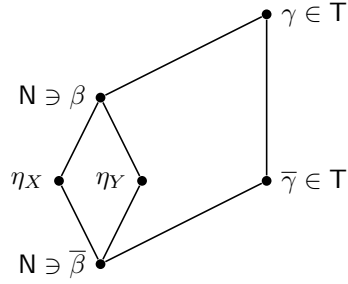
\begin{figure}[htbp]
\begin{center}
\begin{tikzpicture}[
  scale=1.1,
  dot/.style={circle, fill=black, inner sep=1.3pt},
  line/.style={line width=0.6pt}
]

% nodes
\node[dot,label=right:$\gamma\in \Tpos$] (gamma) at (2,3) {};
\node[dot,label=left:$\Npos\ni\beta$] (beta) at (0,2) {};
\node[dot,label=left:$\eta_X$] (etaX) at (-0.5,1) {};
\node[dot,label=left:$\eta_Y$] (etaY) at (0.5,1) {};
\node[dot,label=left:$\Npos\ni\overline{\beta}$] (betabar) at (0,0) {};
\node[dot,label=right:$\overline{\gamma}\in \Tpos$] (gammabar) at (2,1) {};

% cover relations
\draw[line] (gamma) -- (beta);
\draw[line] (gamma) -- (gammabar);
\draw[line] (beta) -- (etaX);
\draw[line] (etaX) -- (betabar);
\draw[line] (beta) -- (etaY);
\draw[line] (etaY) -- (betabar);
\draw[line] (gammabar) -- (betabar);

\end{tikzpicture}
\end{center}
\caption{Relationships among the elements appearing in the proof of \Cref{prop:strong_acyclicity}}
\label{fig:main_Hasse}
\end{figure}

\begin{proof}[Proof of \Cref{prop:target}]
Take $q=\one$.  Since $\A$ is central, $\one$ is a chamber sign
vector.  If $\alpha\geq\one$, then
\[
        \alpha\wedge q=\one\in\Npos.
\]
\Cref{prop:strong_acyclicity} therefore gives the acyclicity of
$\Npos_{\leq\alpha}$.
\end{proof}

\subsection{Conclusion of the proof}
Let $\alpha\geq\one$.
By \Cref{prop:target}, the poset
$\Npos_{A,\leq\alpha}$ is acyclic for every base chamber $A$.
By the universal coefficient theorem, we obtain $\widetilde H^\ast(\Delta_\alpha^{\Npos_A})=0$.
Hence \Cref{lem:T_N_alexander} gives
\[
        H_k(\Delta_\alpha^{\Tpos_A},
             \Delta_\alpha^{\Tpos_A}\cap\partial\Delta_\alpha)=0
\]
for all $k$ and all $A$.
By \Cref{lem:nonspecial_to_poset}, each summand $N_{\ast,\alpha}(\Tope(\A))_{A\to \bullet}$ of $N_{\ast,\alpha}(\Tope(\A))$ is acyclic and hence $N_{\ast,\alpha}(\Tope(\A))$ itself is acyclic.
In particular, the direct summand $N_{\ast,\alpha}(\Tope(\A))_{\bullet\to B}$ of $N_{\ast,\alpha}(\Tope(\A))$ is acyclic.
Therefore the quotient map
\[
    \MC_{\ast,\alpha}(\Tope(\A))_{\bullet\to B}\twoheadrightarrow Sp_{\ast,\alpha}(\Tope(\A))_{\bullet\to B}
\]
is a quasi-isomorphism.
Combining this with
\Cref{prop:special_homology} proves \Cref{thm:periodicity}.

\section{Magnitude homology of tope graphs}\label{sec:affine_closed}

We determine the magnitude homology of the tope graphs of affine hyperplane arrangements.

\subsection{Face flags}
Let $\A$ be an affine hyperplane arrangement and $B\in \Ch(\A)$.
A \emph{face flag} in $\A$ relative to $B$ is a finite weakly decreasing chain
\[
        \mathscr F=(F_1\geq F_2\geq\cdots\geq F_m),
        \quad F_i\in\Face(\A)_{<B},
\]
where $m\geq0$.
Equalities among the $F_i$ are allowed.
Define
\[
    \rank(\mathscr F):=
    \sum_{i=1}^m \rank(\A_{F_i}),
    \quad
    \ell(\mathscr F):=
    \sum_{i=1}^m \#\A_{F_i},
    \quad\text{and}\quad
    \vec{\ell}(\mathscr F):=
    \sum_{i=1}^m \one_{\A_{F_i}}\in\NN^{\oplus \A},
\]
where $\one_{\A_{F_i}}$ denotes the characteristic vector of the subset $\A_{F_i}\subseteq\A$.
Empty sums are understood to be zero.
The pair $(\rank(\mathscr{F}),\ell(\mathscr{F}))$ is called the \emph{bidegree} of $\mathscr{F}$.
For $\alpha\in\NN^{\oplus \A}$ and $k,\ell\geq0$, put
\begin{align*}
        \Flag_{k,\ell}(\A)_B
        &:=\{\mathscr F\mid
              \rank(\mathscr F)=k,
              \ \ell(\mathscr F)=\ell\},\\
        \Flag_{k,\alpha}(\A)_B
        &:=\{\mathscr F\mid
              \rank(\mathscr F)=k,
              \ \vec{\ell}(\mathscr F)=\alpha\},\\
        \Flag_{k,\ell}^\circ(\A)_B
        &:=\{\mathscr F\mid
              \rank(\mathscr F)=k,
              \ \ell(\mathscr{F})=\ell,
              \ \supp(\vec{\ell}(\mathscr F))=\A\}.
\end{align*}

\begin{lemma}\label{lem:flag_support}
Let $\A$ be an affine hyperplane arrangement and let $B\in \Ch(\A)$.
Let $\zero\neq\alpha\in\NN^{\oplus\A}$.
Then, for any $\mathscr F=(F_1\geq\cdots\geq F_m)\in \Flag_{k,\alpha}(\A)_B$, we have
\[
    \A_{F_m}=\supp(\alpha).
\]
\end{lemma}

\begin{proof}
Since $F_1\geq\cdots\geq F_m$, the corresponding subsets of hyperplanes form an increasing chain
\[
        \A_{F_1}\subseteq\cdots\subseteq\A_{F_m}.
\]
The support of $\alpha$ is therefore the union of these subsets, which is $\A_{F_m}$.
\end{proof}

\begin{lemma}[Localization for face flags]\label{lem:flag_inductive}
Let $\A$ be an affine hyperplane arrangement and let $B\in \Ch(\A)$.
Let $\zero\neq\alpha\in\NN^{\oplus \A}$ and put $\A_{\alpha}=\supp(\alpha)$.
If there exists a face $F\in \Face(\A)_{<B}$ with $\A_F=\A_{\alpha}$, then there is a bijection
\[
        \Flag_{k,\alpha}(\A)_B
        \cong
        \Flag_{k,\alpha|_{\A_\alpha}}(\A_\alpha)_{B_\alpha},
\]
where $B_\alpha$ is the unique chamber of $\A_\alpha$ containing $B$.
Otherwise, we have
\[
        \Flag_{k,\alpha}(\A)_B=\varnothing.
\]
\end{lemma}

\begin{proof}
Suppose first that there exists a face $F\in \Face(\A)_{<B}$ with $\A_F=\A_\alpha$.
In this case, there is a bijection
\[
\iota\colon \Face(\A)_{\geq F,<B}\xrightarrow{\sim}\Face(\A_\alpha)_{<B_\alpha}
\]
sending a face $E$ to the unique face of $\A_\alpha$ containing $E$.
If $\mathscr{F}=(F_1\geq \cdots\geq F_m)\in \Flag_{k,\alpha}(\A)_B$, then \Cref{lem:flag_support} shows that $\A_{F_m}=\A_{\alpha}=\A_F$.
Both $F$ and $F_m$ lie below $B$ and have the same zero set
$\A_F=\A_{F_m}$; on every remaining hyperplane their signs agree with the sign of
$B$.  Hence their sign vectors coincide and $F_m=F$.
Therefore any element of $\Flag_{k,\alpha}(\A)_B$ consists of faces from $\Face(\A)_{\geq F,<B}$.
Thus $\iota$ induces the claimed bijection of face flags.

Suppose next that there is no such $F$.
If there is an element $\mathscr{F}=(F_1\geq \cdots\geq F_m)\in \Flag_{k,\alpha}(\A)_B$, then $F_m\in \Face(\A)_{<B}$, and \Cref{lem:flag_support} shows that $\A_{F_m}=\A_{\alpha}$, which contradicts our assumption.
Therefore $\Flag_{k,\alpha}(\A)_B$ is empty.
\end{proof}

\begin{lemma}[Periodicity for face flags]\label{lem:flag_periodicity}
    Let $\A$ be a central hyperplane arrangement of rank $r$, and fix a chamber $B\in \Ch(\A)$.
    Let $\alpha\in\NN^{\oplus \A}$ satisfy $\supp(\alpha)=\A$ (equivalently, $\alpha\geq\one$).
    Then there is a bijection
    \[
        \Flag_{k,\alpha}(\A)_B
        \cong
        \Flag_{k-r,\alpha-\one}(\A)_B,
    \]
    where the right-hand side is interpreted as empty for $k<r$.
\end{lemma}

\begin{proof}
If $\A=\varnothing$, then $r=0$, $\alpha=\zero$, and the identity map gives the
claimed bijection.  We henceforth assume that $\A$ is nonempty.
Let $\mathscr{F}=(F_1\geq \cdots\geq F_m)\in \Flag_{k,\alpha}(\A)_B$.
Then \Cref{lem:flag_support} shows that $\A_{F_m}=\supp(\alpha)=\A$ and hence $F_m$ is the center $F_\A$.
Deleting $F_m=F_\A$ gives a face flag $(F_1\geq \cdots\geq F_{m-1})\in \Flag_{k-r,\alpha-\one}(\A)_B$.
Conversely, given a face flag in $\Flag_{k-r,\alpha-\one}(\A)_B$, we may append $F_\A$ as the last face to obtain a face flag in $\Flag_{k,\alpha}(\A)_B$.
\end{proof}

\subsection{Magnitude homology and face flags}
As we have seen above, the magnitude homology and the set of face flags of a hyperplane arrangement satisfy exactly the same recurrence.
This yields the following theorem.

\begin{theorem}\label{thm:closed}
Let $\A$ be an affine hyperplane arrangement and $B\in \Ch(\A)$.
\begin{enumerate}
    \item $\MH_{k,\alpha}(\Tope(\A))_{\bullet\to B}\cong\ZZ^{\oplus\Flag_{k,\alpha}(\A)_B}$.
    \item $\MH_{k,\ell}(\Tope(\A))_{\bullet\to B}\cong\ZZ^{\oplus\Flag_{k,\ell}(\A)_B}$.
    \item $\MH_{k,\ell}^\circ(\Tope(\A))_{\bullet\to B}\cong\ZZ^{\oplus\Flag_{k,\ell}^\circ(\A)_B}$.
\end{enumerate}
\end{theorem}

\begin{proof}
We first prove (1) by induction on $\lvert\alpha\rvert$.
If $\lvert\alpha\rvert=0$, then $\alpha=\zero$.
We have $P_{0,\zero}(\Tope(\A))_{\bullet\to B}=\{(B)\}$ and $P_{k,\zero}(\Tope(\A))_{\bullet\to B}=\varnothing$ for $k>0$.
Therefore
\[
        \MH_{k,\zero}(\Tope(\A))_{\bullet\to B}
        \cong
        \begin{cases}
        \ZZ,& \text{if }k=0;\\
        0,& \text{if }k>0.
        \end{cases}
\]
On the other hand, we have $\Flag_{0,\zero}(\A)_B=\{()\}$ and $\Flag_{k,\zero}(\A)_B=\varnothing$ for $k>0$, so the claim is proved.

Next, assume that $\lvert\alpha\rvert>0$ and put
$\A_\alpha=\supp(\alpha)\subseteq\A$.
If there is no face $F\in\Face(\A)_{<B}$ with $\A_F=\A_\alpha$, then
\Cref{thm:localization} and \Cref{lem:flag_inductive} give
\[
\MH_{k,\alpha}(\Tope(\A))_{\bullet\to B}=0,
\qquad
\Flag_{k,\alpha}(\A)_B=\varnothing.
\]
Assume that such a face $F$ exists.  Then $\A_\alpha=\A_F$ is central, and
\Cref{thm:localization} and \Cref{lem:flag_inductive} give
\begin{align*}
    \MH_{k,\alpha}(\Tope(\A))_{\bullet\to B}
    &\cong
    \MH_{k,\alpha|_{\A_\alpha}}(\Tope(\A_\alpha))_{\bullet\to B_\alpha},\\
    \Flag_{k,\alpha}(\A)_B
    &\cong
    \Flag_{k,\alpha|_{\A_\alpha}}(\A_\alpha)_{B_\alpha}.
\end{align*}
After replacing $(\A,B,\alpha)$ by
$(\A_\alpha,B_\alpha,\alpha|_{\A_\alpha})$, we may therefore assume that $\A$ is
central and $\supp(\alpha)=\A$.

Let $r=\rank\A$.  The periodicity results \Cref{thm:periodicity} and
\Cref{lem:flag_periodicity} now give
\begin{align*}
    \MH_{k,\alpha}(\Tope(\A))_{\bullet\to B}
    &\cong
    \MH_{k-r,\alpha-\one}(\Tope(\A))_{\bullet\to B},\\
    \Flag_{k,\alpha}(\A)_B
    &\cong
    \Flag_{k-r,\alpha-\one}(\A)_B.
\end{align*}
Since $\A\neq\varnothing$ and
$\lvert\alpha-\one\rvert=\lvert\alpha\rvert-\#\A<\lvert\alpha\rvert$, the
claim follows from the induction hypothesis.

Taking a direct sum over all $\alpha$ with $\lvert \alpha\rvert=\ell$ proves (2).
Similarly, taking a direct sum over all $\alpha$ with $\lvert \alpha\rvert=\ell$ and $\supp(\alpha)=\A$ proves (3).
\end{proof}

\begin{corollary}
    Let $\A$ be an affine hyperplane arrangement.
    Assume that $\Tope(\A)$ is of finite type.
    Then the unique magnitude weighting on $\Tope(\A)$ is given by
    \[
    \wt(B)=\sum_{\ell=0}^\infty\sum_{k=0}^\ell(-1)^k\#\Flag_{k,\ell}(\A)_B\,q^\ell.
    \]
\end{corollary}

We define $\Flag_{k,\ell}(\A)$, $\Flag_{k,\alpha}(\A)$, and $\Flag_{k,\ell}^\circ(\A)$ by
\begin{align*}
        \Flag_{k,\ell}(\A)&=\coprod_{B\in \Ch(\A)}\Flag_{k,\ell}(\A)_B,\\
        \Flag_{k,\alpha}(\A)&=\coprod_{B\in \Ch(\A)}\Flag_{k,\alpha}(\A)_B,\\
        \Flag_{k,\ell}^\circ(\A)&=\coprod_{B\in \Ch(\A)}\Flag_{k,\ell}^\circ(\A)_B.
\end{align*}

\begin{corollary}\label{cor:closed_finite}
Let $\A$ be an affine hyperplane arrangement.
\begin{enumerate}
    \item $\MH_{k,\alpha}(\Tope(\A))\cong\ZZ^{\oplus\Flag_{k,\alpha}(\A)}$.
    \item $\MH_{k,\ell}(\Tope(\A))\cong\ZZ^{\oplus\Flag_{k,\ell}(\A)}$.
    \item $\MH_{k,\ell}^\circ(\Tope(\A))\cong\ZZ^{\oplus\Flag_{k,\ell}^\circ(\A)}$.
\end{enumerate}
\end{corollary}

\begin{corollary}\label{cor:flag_number_from_LA}
    For a finite affine hyperplane arrangement $\A$, the magnitude Betti numbers
    \[
    \beta_{k,\ell}(\Tope(\A))=\rank\MH_{k,\ell}(\Tope(\A))
    \]
    are uniquely determined by the intersection poset $L(\A)$.
\end{corollary}

\begin{proof}
Let $\A'$ be another finite affine hyperplane arrangement and let
\(\varphi\colon L(\A)\xrightarrow{\sim}L(\A')\) be a poset isomorphism.  Since
hyperplanes are the atoms of the intersection poset, $\varphi$ induces a bijection
between $\A$ and $\A'$.  For $\alpha\in\NN^{\oplus\A}$, let
$\varphi_\ast\alpha\in\NN^{\oplus\A'}$ be the vector obtained by transporting its
coordinates along this bijection.  We prove, by lexicographic induction on
\((\#\A,\lvert\alpha\rvert)\), that
\[
    \#\Flag_{k,\alpha}(\A)
    =
    \#\Flag_{k,\varphi_\ast\alpha}(\A').
\]
This stronger statement implies the corollary after summing over all vectors of total
length $\ell$.

If $\alpha=\zero$, then the only flags have $k=0$, and their number is
$\#\Ch(\A)$.  Zaslavsky's theorem \cite{Zaslavsky1975} expresses this chamber
number in terms of $L(\A)$, so the assertion holds.

Assume $\alpha\neq\zero$ and put $S=\supp(\alpha)$.  First suppose that
$S=\A$.  Centrality is equivalent to $L(\A)$ having a greatest element, and,
in the central case, $r=\rank\A$ is the graded rank of that element.  Thus both
centrality and $r$ are determined by the intersection poset.  If $\A$ is noncentral,
then no face $F$ satisfies $\A_F=\A$;
by \Cref{lem:flag_support}, both flag sets in question are empty.  If $\A$ is central,
then so is $\A'$, with the same rank $r$, and \Cref{lem:flag_periodicity} gives
\[
\Flag_{k,\alpha}(\A)
\cong
\Flag_{k-r,\alpha-\one}(\A),
\qquad
\Flag_{k,\varphi_\ast\alpha}(\A')
\cong
\Flag_{k-r,\varphi_\ast(\alpha-\one)}(\A').
\]
If $k<r$, both right-hand flag sets are empty by convention.  If $k\geq r$, the
second component of the induction parameter has decreased, so the desired equality
follows from the induction hypothesis.

Now suppose that $S\subsetneq\A$.  Whether there exists a flat $X\in L(\A)$
with $\A_X=S$ is determined by $L(\A)$.  If no such flat exists, then
\Cref{lem:flag_support} shows that both flag sets are empty.  Otherwise let $X$ be
that flat and put $X'=\varphi(X)$.  Faces $F$ with $\A_F=S$ are precisely the
chambers of the restriction $\A^X$.  For each such face, the chambers $B$ satisfying
$F<B$ are in bijection with the chambers of the localization $\A_X$, via
$B\mapsto B_X$.  Applying \Cref{lem:flag_inductive} for each pair $(F,B)$ gives
\[
    \#\Flag_{k,\alpha}(\A)
    =
    \#\Ch(\A^X)\,
    \#\Flag_{k,\alpha|_{S}}(\A_X).
\]
The analogous formula holds for $\A'$ and $X'$.  The intersection posets
$L(\A_X)$ and $L(\A^X)$ are respectively the interval
$[\widehat{0},X]$ and the principal filter $L(\A)_{\geq X}$, and similarly for
$\A'$.  Hence $\varphi$ identifies the corresponding localization and restriction
posets.  Zaslavsky's theorem gives
$\#\Ch(\A^X)=\#\Ch((\A')^{X'})$, while
$\#\A_X=\#S<\#\A$ allows the induction hypothesis to be applied to the
localizations.  The displayed products are therefore equal.

Finally, the atom bijection induced by $\varphi$ identifies the sets of vectors
$\alpha$ of any fixed total length $\ell$.  By \Cref{cor:closed_finite},
\[
\beta_{k,\ell}(\Tope(\A))
=
\sum_{\substack{\alpha\in \NN^{\oplus\A}\\\lvert\alpha\rvert=\ell}}\#\Flag_{k,\alpha}(\A),
\]
and the same formula holds for $\A'$.  Thus the magnitude Betti numbers are
determined by $L(\A)$.
\end{proof}

\begin{corollary}[Homological reciprocity {\cite[Conjecture 7.1 (6)]{Koizumi_Liu}}]\label{reciprocity}
Let $\A$ be a central hyperplane arrangement of rank $r$, and put $n=\#\A$.
Then
\[
        \MH_{k,\ell}^{\circ}(\Tope(\A))
        \cong
        \MH_{k-r,\ell-n}(\Tope(\A)),
\]
where the right-hand side is interpreted as zero for $k<r$ or $\ell<n$.
\end{corollary}

\begin{proof}
By \Cref{cor:closed_finite}, it suffices to show that
\[
        \Flag_{k,\ell}^{\circ}(\A)_B
        \cong
        \Flag_{k-r,\ell-n}(\A)_B
\]
for any $B\in \Ch(\A)$.
The left-hand side can be decomposed as
\[
        \Flag_{k,\ell}^{\circ}(\A)_B
        =
        \coprod_{\substack{|\alpha|=\ell\\ \supp(\alpha)=\A}}
        \Flag_{k,\alpha}(\A)_B.
\]
If $\ell<n$, there is no vector $\alpha\in\NN^{\oplus \A}$ with $|\alpha|=\ell$ and $\supp(\alpha)=\A$, so this is empty.
Assume $\ell\geq n$.
Then the assignment $\alpha\mapsto \alpha-\one$ is a bijection from the set of vectors $\alpha\in\NN^{\oplus \A}$ with $|\alpha|=\ell$ and $\supp(\alpha)=\A$ to the set of all vectors $\beta\in\NN^{\oplus \A}$ with $|\beta|=\ell-n$.
Applying \Cref{lem:flag_periodicity} to each summand gives
\[
        \Flag_{k,\alpha}(\A)_B
        \cong
        \Flag_{k-r,\alpha-\one}(\A)_B.
\]
Summing over all $\alpha$ with $\supp(\alpha)=\A$ gives the stated isomorphism.
\end{proof}

\subsection{Interpretation via Stanley--Reisner rings}\label{sec:stanley_reisner}
Let $K$ be a field and let $\Delta$ be a simplicial complex on a vertex set $V$.
The \emph{Stanley--Reisner ideal} of $\Delta$ is the squarefree monomial ideal
\[
    I_\Delta
    =
    \biggl(
        \prod_{v\in\sigma}x_v
        \mathrel{\Big|}
        \sigma\subseteq V\text{ finite},\ \sigma\notin\Delta
    \biggr)
    \subseteq K[x_v\mid v\in V],
\]
and the \emph{Stanley--Reisner ring} of $\Delta$ is
\[
    K[\Delta]:=K[x_v\mid v\in V]/I_\Delta.
\]
For a poset $P$, the Stanley--Reisner ring of its order complex $\Delta P$ can be written as
\[
    K[\Delta P]
    =
    K[x_p\mid p\in P]
    \big/
    (x_px_{p'}\mid p,p'\in P\text{ are incomparable}).
\]

Fix an affine hyperplane arrangement $\A$ and a chamber $B\in\Ch(\A)$, and set
\[
    P_B:=\Face(\A)_{<B},
    \qquad
    R_{\A,B}:=K[\Delta P_B].
\]
Endow $R_{\A,B}$ with the $\NN^2$-grading determined by
\[
    \deg(x_F)=(\rank(\A_F),\#\A_F).
\]
The defining ideal of $R_{\A,B}$ is bihomogeneous for this grading.

\begin{proposition}\label{prop:stanley_reisner}
Let $\A$ be an affine hyperplane arrangement such that $\Tope(\A)$ is of finite type, and let $B\in \Ch(\A)$.
For every $k,\ell\geq0$, we have
\[
    \dim_{K}(R_{\A,B})_{k,\ell}
    =
    \#\Flag_{k,\ell}(\A)_B
    =
    \rank\MH_{k,\ell}(\Tope(\A))_{\bullet\to B}.
\]
Consequently, the bigraded Hilbert series of $R_{\A,B}$ satisfies
\[
    \operatorname{Hilb}(R_{\A,B};u,q)
    :=
    \sum_{k,\ell\geq0}
    \dim_{K}(R_{\A,B})_{k,\ell}u^kq^\ell
    =
    \sum_{k,\ell\geq0}
    \rank\MH_{k,\ell}(\Tope(\A))_{\bullet\to B}u^kq^\ell.
\]
\end{proposition}

\begin{proof}
The residue classes of the monomials not contained in $I_{\Delta P_B}$ form a $K$-basis of $R_{\A,B}$.
Such a monomial has chain support, so its variables can be arranged uniquely in weakly decreasing order as
\[
    x_{F_1}x_{F_2}\cdots x_{F_m},
    \qquad
    F_1\geq F_2\geq\cdots\geq F_m,
\]
where repetitions are allowed.
This gives a bijection between the monomial basis of $R_{\A,B}$ and the face flags relative to $B$; the monomial $1$ corresponds to the empty face flag.
For the corresponding face flag $\mathscr F=(F_1\geq\cdots\geq F_m)$, we have
\[
    \deg(x_{F_1}\cdots x_{F_m})
    =
    \left(
        \sum_{i=1}^m\rank(\A_{F_i}),
        \sum_{i=1}^m\#\A_{F_i}
    \right)
    =
    (\rank(\mathscr F),\ell(\mathscr F)).
\]
Thus the first equality follows, and the second is \Cref{thm:closed}.
\end{proof}

\begin{remark}
    When $\A$ is finite, $\Ch(\A)$ and each $P_B$ are finite, and \eqref{vertexdecomp} and \Cref{prop:stanley_reisner} give
    \[
        \sum_{k,\ell\geq0}
        \rank\MH_{k,\ell}(\Tope(\A))u^kq^\ell
        =
        \sum_{B\in\Ch(\A)}
        \operatorname{Hilb}(R_{\A,B};u,q).
    \]
\end{remark}

\subsection{Coxeter arrangements}
In the case of a Coxeter arrangement, the magnitude homology of the tope graph can be described in terms of the Coxeter system.

Let $(W,S)$ be a Coxeter system of finite or affine type.
A subset $T\subseteq S$ is called \emph{spherical} if the corresponding standard parabolic subgroup $W_T=\langle T\rangle$ is finite.

\begin{lemma}\label{lem:spherical}
    Let $(W,S)$ be a Coxeter system of finite or affine type.
    Let $\A$ be the associated Coxeter arrangement with base chamber $B$.
    Let $H_s$ denote the hyperplane corresponding to $s\in S$.
    Then there exists an order-reversing bijection
    \[
    \{\varnothing\neq T\subseteq S\mid \text{spherical}\}\xrightarrow{\sim}\Face(\A)_{<B};\quad T\mapsto F_T:=\overline{B}\cap \bigcap_{t\in T} H_t.
    \]
    Moreover, the localized arrangement $\A_{F_T}$ can be identified with the Coxeter arrangement associated to $(W_T,T)$.
    In particular, we have $\rank(\A_{F_T})=\#T$ and $\#\A_{F_T}=\ell(w_0^T)$, where $w_0^T$ is the longest element in the finite Coxeter group $W_T$.
\end{lemma}

\begin{proof}
    If $(W,S)$ is of finite type, $\overline B$ is a simplicial cone and every subset of $S$ is spherical, so its faces are precisely the sets $F_T$.
    If $(W,S)$ is of affine type, the nerve $L(W,S)$, whose simplices are the nonempty spherical subsets of $S$, is isomorphic to the boundary of the dual polytope of $\overline B$; see \cite[Example~7.1.4 and Corollary~8.2.10]{Davis2008}.
    Consequently, in either case, $F_T$ is a nonempty face of $\overline B$ if and only if $T$ is spherical.
    Moreover,
    \[
        T(F):=\{s\in S\mid F\subseteq H_s\}
    \]
    satisfies $T(F_T)=T$, and hence $T\mapsto F_T$ is a bijection from the nonempty spherical subsets of $S$ to the proper faces of $B$.

    Fix a spherical subset $T$, set $F=F_T$, and choose a point $x$ in the relative interior of $F$.
    The walls of $B$ containing $x$ are exactly the $H_t$ with $t\in T$.
    The isotropy subgroup of $x$ is $W_x=W_T$; see \cite[\S5.2]{Davis2008}.
    Let
    \[
        R:=\{wsw^{-1}\mid w\in W,\ s\in S\}
        \quad\text{and}\quad
        R_T:=\{utu^{-1}\mid u\in W_T,\ t\in T\}
    \]
    be the sets of reflections of $(W,S)$ and $(W_T,T)$, respectively.  For
    $r\in R$, denote its reflecting hyperplane by $H_r$.  Since
    $x$ is in the relative interior of $F$, a hyperplane contains $F$ if and only if it contains $x$.
    Thus
    \[
    \begin{aligned}
        \A_F
        &=\{H_r\mid r\in R,\ r(x)=x\}\\
        &=\{H_r\mid r\in R\cap W_x\}\\
        &=\{H_r\mid r\in R\cap W_T\}.
    \end{aligned}
    \]
    The inclusion $R_T\subseteq R\cap W_T$ is immediate, while
    \cite[Lemma~4.2.3]{Davis2008} gives the reverse inclusion.
    Hence $R_T=R\cap W_T$ and consequently
    \[
        \A_F=\{H_r\mid r\in R_T\}.
    \]
    After translating $x$ to the origin and passing to the normal space of $F$, these are exactly the reflecting hyperplanes of the finite Coxeter system $(W_T,T)$.
    Therefore $\A_F$ has rank $\#T$ and $\#\A_F=\#R_T$.
    Applying \cite[Lemma~4.6.1(ii)]{Davis2008} to the finite Coxeter system
    $(W_T,T)$ gives $\#R_T=\ell(w_0^T)$.
    Hence $\#\A_{F_T}=\ell(w_0^T)$, as required.
\end{proof}

\begin{definition}
    Let $(W,S)$ be a Coxeter system of finite or affine type.
    A \emph{spherical flag} of $(W,S)$ is a sequence of nonempty spherical subsets
    \[
    \mathscr{F}=(T_1\subseteq T_2\subseteq\cdots\subseteq T_m),
    \]
    where $m\geq 0$.
    Define
    \[
    \rank(\mathscr{F}):=\sum_{i=1}^m \#T_i,\quad
    \ell(\mathscr{F}) :=\sum_{i=1}^m \ell(w_0^{T_i}).
    \]
    We write $\Flag_{k,\ell}(W,S)$ for the set of spherical flags $\mathscr{F}$ of $(W,S)$ with $\rank(\mathscr{F})=k$ and $\ell(\mathscr{F})=\ell$.
\end{definition}

\begin{corollary}\label{cor:coxeter}
    Let $(W,S)$ be a Coxeter system of finite or affine type, and let $\Cay(W,S)$
    be its Cayley graph.  Then, for any $b\in W$, we have
    \begin{align*}
    \MH_{k,\ell}(\Cay(W,S))_{\bullet\to b}
    &\cong
    \mathbb{Z}^{\oplus \Flag_{k,\ell}(W,S)},\\
    \MH_{k,\ell}(\Cay(W,S))
    &\cong
    \mathbb{Z}^{\oplus(\Flag_{k,\ell}(W,S)\times W)}.
    \end{align*}
\end{corollary}

\begin{proof}
    Since $\Cay(W,S)$ is vertex-transitive, we may assume that $b$ is the identity
    element.  Let $\A$ be the Coxeter arrangement associated to $(W,S)$.  The Cayley
    graph $\Cay(W,S)$ is isomorphic to $\Tope(\A)$, so the first assertion follows
    from \Cref{thm:closed} and \Cref{lem:spherical}.  Taking the direct sum over all
    $b\in W$ in the vertex decomposition \eqref{vertexdecomp} gives the second
    assertion.
\end{proof}

\begin{remark}
    Let $(W,S)$ be a Coxeter system of finite or affine type.
    By \Cref{topecoxeter}, the unique magnitude weighting on $\Cay(W,S)$ is given by the reciprocal of the Poincar\'e series $\sum_{w\in W}q^{\ell(w)}$.
    Comparing this with \Cref{lem:weight} and \Cref{cor:coxeter}, we obtain a nontrivial identity
    \[
    \frac{1}{\sum_{w\in W}q^{\ell(w)}}=\sum_{\ell=0}^\infty\sum_{k=0}^\ell(-1)^k\#\Flag_{k,\ell}(W,S)\,q^\ell.
    \]
    This identity is mentioned by Dyer \cite{Dyer2024} as a formula obtained by recursively expanding the Steinberg formula. It is also observed there that this formula can be interpreted as a special case of a formula describing the Hilbert series of the Stanley--Reisner ring of the order complex of a lower Eulerian poset. Our result may be regarded as a homological refinement of this formula.
\end{remark}

\section{Examples}

This section illustrates the main theorems (\Cref{thm:closed}, \Cref{cor:closed_finite}) by computing the magnitude homology of several tope graphs.
In particular, the Cayley graphs of $S_n$ provide a new family of examples for which the magnitude homology can be computed completely.

\subsection{Path graphs}
Let $n\geq 1$.
Let $\A=\{H_1,\dots,H_n\}$ be the affine hyperplane arrangement in $\mathbb R$ given by
$H_i=\{i\}$ for $1\leq i\leq n$.
It has $n+1$ chambers $C_0,\dots,C_n$, where $H_i$ lies between $C_{i-1}$ and $C_i$.
The tope graph $\Tope(\A)$ is the path graph $P_{n+1}$.
We have $\Flag_{k,\ell}(\A)_{C_i}=\varnothing$ for $k\neq \ell$, and $\Flag_{0,0}(\A)_{C_i}=\{()\}$.
For $(k,\ell)=(m,m)$ with $m\geq 1$, we have
\[
\Flag_{m,m}(\A)_{C_i}=
\begin{cases}
    \{(\underbrace{H_1\geq\cdots\geq H_1}_m)\},&\text{if }i=0;\\
    \{(\underbrace{H_i\geq\cdots\geq H_i}_m),(\underbrace{H_{i+1}\geq\cdots\geq H_{i+1}}_m)\},&\text{if }1\leq i\leq n-1;\\
    \{(\underbrace{H_n\geq\cdots\geq H_n}_m)\},&\text{if }i=n.
\end{cases}
\]
By \Cref{cor:closed_finite}, we obtain
\[
        \beta_{k,\ell}(P_{n+1})=
        \begin{cases}
            0,&\text{if }k\neq\ell;\\
            n+1,&\text{if }k=\ell=0;\\
            2n,&\text{if }k=\ell\geq 1.
        \end{cases}
\]
This agrees with the computation of the magnitude homology for trees given in \cite[Corollary 6.8]{Hepworth2017}.

\subsection{Even cycles}
The magnitude homology of the cycle graph $C_n$ was determined by Gu \cite{Gu2018}.
Asao--Wakatsuki constructed explicit magnitude cycles forming a basis \cite[Theorem 8.2]{Asao_Wakatsuki}.
Here, we recover Gu's computation for even cycles by applying \Cref{cor:closed_finite}.

\begin{definition}
    Let $n\geq 3$.
    Define $T_n(k,\ell)$ for $k,\ell\in \ZZ$ as follows.
    \begin{enumerate}
        \item $T_n(k,\ell)=0$ if $k<0$ or $\ell<0$.
        \item $T_n(0,0)=2n$ and $T_n(1,1)=4n$.
        \item $T_n(k,\ell)=\max\{T_n(k-1,\ell-1), T_n(k-2,\ell-n)\}$ for $(k,\ell)\notin\{(0,0),(1,1)\}$.
    \end{enumerate}
\end{definition}

\begin{theorem}[{\cite[Theorem 4.8]{Gu2018}}]
    For any $n\geq 3$, $\MH_{k,\ell}(C_{2n})$ is a free abelian group of rank $T_n(k,\ell)$.
\end{theorem}

\begin{proof}
    The even cycle $C_{2n}$ can be regarded as the tope graph of the hyperplane arrangement $U_{2,n}$ in $\mathbb{R}^2$ consisting of $n$ lines through the origin. 
    By \Cref{cor:closed_finite}, it suffices to show that $\#\Flag_{k,\ell}(U_{2,n})=T_n(k,\ell)$.
    The face poset of the central hyperplane arrangement $U_{2,n}$ can be described as follows.
    There are $2n$ chambers $C_1,\dots,C_{2n}$, $2n$ rank-one faces $F_1,\dots,F_{2n}$, and the center $\{0\}$.
    The partial order on these faces is generated by
    \[
    \{0\}\leq F_i,\quad F_i\leq C_i,\quad F_{i+1}\leq C_i,
    \]
    where $F_{2n+1}=F_1$.
    The bidegree of $F_i$ is $(1,1)$, and the bidegree of $\{0\}$ is $(2,n)$.
    Therefore a face flag in $U_{2,n}$ relative to $C_i$ is of the form
    \[
    (\underbrace{F_j\geq \cdots\geq F_j}_a
    \geq
    \underbrace{\{0\}\geq \cdots\geq \{0\}}_b),
    \]
    where $j\in\{i,i+1\}$ when $a>0$; if $a=0$, no choice of $j$ is present.
    Such a flag has bidegree
    \[
    (k,\ell)=(a+2b,a+nb).
    \]
    In particular, $\#\Flag_{k,\ell}(U_{2,n})$ is given by
    \[
    \#\Flag_{k,\ell}(U_{2,n})=\begin{cases}
        2n,&\text{if }(k,\ell)=(2b,nb)\text{ for some }b\geq 0;\\
        4n,&\text{if }(k,\ell)=(a+2b,a+nb)\text{ for some }a>0, b\geq 0;\\
        0,&\text{otherwise}.
    \end{cases}
    \]
    These cardinalities satisfy precisely the recurrence defining $T_n$.
    Thus $\#\Flag_{k,\ell}(U_{2,n})=T_n(k,\ell)$.
\end{proof}

\subsection{Cayley graph of $S_n$}
Let $n\geq 2$.
The Coxeter arrangement of type $A_{n-1}$ is also known as the braid arrangement $\Br_n$.
It is an arrangement in $\mathbb R^n$ defined by the hyperplanes $x_i=x_j$ for $1\le i<j\le n$.
The tope graph of $\Br_n$ is isomorphic to the Cayley graph of $S_n$ with respect to adjacent transpositions, which we denote by $\Cay(S_n,\adj)$.
For example, we have $\Cay(S_2,\adj)\cong K_2$ and $\Cay(S_3,\adj)\cong C_6$.
\Cref{fig:B_4} shows $\Cay(S_4,\adj)$; its vertices correspond to the $24$ chambers of $\Br_4$.

\begin{figure}[htbp]
\centering
\tdplotsetmaincoords{80}{115}
\begin{tikzpicture}[tdplot_main_coords, scale=1.0, line join=round, line cap=round]
\foreach \name/\x/\y/\z in {
  v0/0/-1/-2,
  v1/0/-1/ 2,
  v2/0/ 1/-2,
  v3/0/ 1/ 2,
  v4/0/-2/-1,
  v5/0/-2/ 1,
  v6/0/ 2/-1,
  v7/0/ 2/ 1,
  v8/-1/0/-2,
  v9/-1/0/ 2,
  v10/1/0/-2,
  v11/1/0/ 2,
  v12/-2/0/-1,
  v13/-2/0/ 1,
  v14/2/0/-1,
  v15/2/0/ 1,
  v16/-1/-2/0,
  v17/-1/ 2/0,
  v18/1/-2/0,
  v19/1/ 2/0,
  v20/-2/-1/0,
  v21/-2/ 1/0,
  v22/2/-1/0,
  v23/2/ 1/0}{
  \coordinate (\name) at (\x,\y,\z);
}
\foreach \a/\b in {
  v0/v4,  v0/v8,  v0/v10,
  v1/v5,  v1/v9,  v1/v11,
  v2/v6,  v2/v8,  v2/v10,
  v3/v7,  v3/v9,  v3/v11,
  v4/v16, v4/v18,
  v5/v16, v5/v18,
  v6/v17, v6/v19,
  v7/v17, v7/v19,
  v8/v12,
  v9/v13,
  v10/v14,
  v11/v15,
  v12/v20, v12/v21,
  v13/v20, v13/v21,
  v14/v22, v14/v23,
  v15/v22, v15/v23,
  v16/v20,
  v17/v21,
  v18/v22,
  v19/v23}{
  \draw[line width=0.6pt] (\a) -- (\b);
}
\foreach \v in {
  v0,v1,v2,v3,v4,v5,v6,v7,v8,v9,v10,v11,
  v12,v13,v14,v15,v16,v17,v18,v19,v20,v21,v22,v23}{
  \fill (\v) circle (1.3pt);
}
\end{tikzpicture}
\caption{Cayley graph $\Cay(S_4,\adj)$}
\label{fig:B_4}
\end{figure}
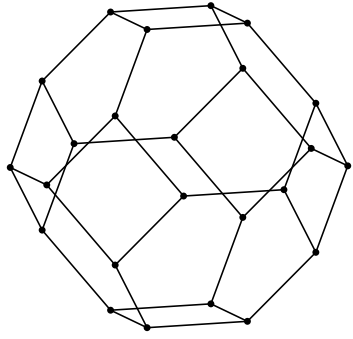

\begin{proposition}
For any $n\geq 2$, $\MH_{k,\ell}(\Cay(S_n,\adj))$ is a free abelian group of rank $n!N_n(k,\ell)$, where
\[
  N_n(k,\ell)
  =
  \#\biggl\{
  (e_1,\ldots,e_{n-1})\in \NN^{n-1}
  \biggm|
  \sum_{i=1}^{n-1} e_i=k,\ 
  \sum_{1\leq a<b\leq n}
  \min(e_a,e_{a+1},\ldots,e_{b-1})=\ell
  \biggr\}.
\]
\end{proposition}

\begin{proof}
Write $\adj=\{s_1,s_2,\dots,s_{n-1}\}$, where $s_j=(j\ j+1)\in S_n$.
By \Cref{cor:coxeter}, it suffices to show that $\#\Flag_{k,\ell}(S_n,\adj)=N_n(k,\ell)$.
An element of $\Flag_{k,\ell}(S_n,\adj)$ is a chain
\[
\mathscr{F}=(T_1\subseteq T_2\subseteq\cdots\subseteq T_m)
\]
of nonempty subsets of $\adj$.
Recall that
\[
    \rank(\mathscr{F})=\sum_{i=1}^m\#T_i,\quad \ell(\mathscr{F})=\sum_{i=1}^m\ell(w_0^{T_i}),
\]
where $w_0^{T_i}$ is the longest element of the subgroup $\langle T_i\rangle\subseteq S_n$.
If $T_i=R_1\sqcup \cdots \sqcup R_s$ is the decomposition into connected blocks, then we have
\[
    \ell(w_0^{T_i})=\sum_{j=1}^s\binom{\#R_j+1}{2}=\#\{(a,b)\mid 1\leq a< b\leq n,\ \{s_a,s_{a+1},\dots,s_{b-1}\}\subseteq T_i\}.
\]
Therefore $\ell(\mathscr{F})$ is given by
\[
    \ell(\mathscr{F})=\sum_{1\leq a<b\leq n} \#\{i\mid \{s_a,s_{a+1},\dots,s_{b-1}\}\subseteq T_i\}.
\]

Define $e_j=\#\{i\mid s_j\in T_i\}$ for $1\leq j\leq n-1$.
This gives a vector
$e=(e_1,\ldots,e_{n-1})\in \NN^{n-1}$.
Conversely, every such vector determines a unique spherical flag.  Indeed, if
$m=\max_j e_j$, then one recovers $T_i$ by
\[
  T_i=\{j\mid e_j\geq m-i+1\,\},
  \quad 1\leq i\leq m.
\]
The zero vector corresponds to the empty flag.
Hence spherical flags are in bijection with vectors in $\NN^{n-1}$.
Under this bijection, we have
\[
    \rank(\mathscr{F})
    =
    \sum_{j=1}^{n-1} e_j,
    \quad
    \ell(\mathscr{F})
    =
    \sum_{1\leq a<b\leq n}
    \min(e_a,e_{a+1},\ldots,e_{b-1}).
\]
This shows that $\#\Flag_{k,\ell}(S_n,\adj)=N_n(k,\ell)$.
\end{proof}

For example, the magnitude Betti numbers of $\Cay(S_4,\adj)$ for $0\leq k,\ell\leq 10$ are given in \Cref{tab:B4}.
These values come from the vectors
$(e_1,e_2,e_3)\in\NN^3$ in the proposition; for instance, the single vector
$(1,1,1)$ contributes to $(k,\ell)=(3,6)$.

\begin{table}[htbp]
\begin{tabular}{c|rrrrrrrrrrr}
$k\backslash \ell$
& 0 & 1 & 2 & 3 & 4 & 5 & 6 & 7 & 8 & 9 & 10 \\
\hline
0  & 24 & 0  & 0  & 0   & 0   & 0   & 0   & 0   & 0   & 0   & 0   \\
1  & 0  & 72 & 0  & 0   & 0   & 0   & 0   & 0   & 0   & 0   & 0   \\
2  & 0  & 0  & 96 & 48  & 0   & 0   & 0   & 0   & 0   & 0   & 0   \\
3  & 0  & 0  & 0  & 120 & 96  & 0   & 24  & 0   & 0   & 0   & 0   \\
4  & 0  & 0  & 0  & 0   & 144 & 96  & 48  & 72  & 0   & 0   & 0   \\
5  & 0  & 0  & 0  & 0   & 0   & 168 & 96  & 96  & 96  & 48  & 0   \\
6  & 0  & 0  & 0  & 0   & 0   & 0   & 192 & 96  & 96  & 168 & 96  \\
7  & 0  & 0  & 0  & 0   & 0   & 0   & 0   & 216 & 96  & 96  & 240 \\
8  & 0  & 0  & 0  & 0   & 0   & 0   & 0   & 0   & 240 & 96  & 96  \\
9  & 0  & 0  & 0  & 0   & 0   & 0   & 0   & 0   & 0   & 264 & 96  \\
10 & 0  & 0  & 0  & 0   & 0   & 0   & 0   & 0   & 0   & 0   & 288
\end{tabular}
\vspace{1em}
\caption{Magnitude Betti numbers of $\Cay(S_4,\adj)$ for $0\leq k,\ell\leq 10$}
\label{tab:B4}
\end{table}

\subsection{$n$-diagonal graphs that are not diagonal}

Let $G$ be a finite connected graph and $n\geq 1$.
We say that $G$ is \emph{$n$-diagonal} if
\[
\forall k\leq n,\quad \forall \ell>k,\quad \MH_{k,\ell}(G)=0.
\]
This means that the magnitude homology concentrates on the diagonal $k=\ell$ in the range $0\leq k\leq n$.
All finite graphs are $1$-diagonal \cite[Example 4.6]{Leinster2021}.
Here, we construct, for every $n\geq 1$, a graph whose magnitude homology is concentrated on the diagonal for $k\leq n$, but is not diagonal in all degrees.

Consider the generic central hyperplane arrangement $U_{n+1,n+2}$ in $\mathbb R^{n+1}$ (\Cref{topegeneric}).
Its tope graph $G_n=\Tope(U_{n+1,n+2})$ is isomorphic to the two-punctured hypercube $\{+,-\}^{n+2}\setminus\{(+,\dots,+),(-,\dots,-)\}$.
The face $\{0\}$ of $U_{n+1,n+2}$ is of bidegree $(n+1,n+2)$, and all other faces are of bidegree $(k,k)$ for some $k\leq n$.
Therefore, for any face flag $\mathscr{F}$, we have either $\rank(\mathscr{F})=\ell(\mathscr{F})$ or $\rank(\mathscr{F})\geq n+1$.
By \Cref{cor:closed_finite}, this shows that $G_n$ is $n$-diagonal.
On the other hand, the face flag $(\{0\})$ is of bidegree $(n+1,n+2)$, so we have $\MH_{n+1,n+2}(G_n)\neq 0$.
In particular, $G_n$ is not diagonal.

\section*{Acknowledgements}
The author thanks Ye Liu for valuable discussions during the early stages of this research, and Yasuhiko Asao and Masahiko Yoshinaga for helpful comments on a draft.

\printbibliography
\end{document}